\documentclass[a4paper,12pt,twoside]{amsart}
\usepackage[T1]{fontenc}
\usepackage{vmargin}
\usepackage{amssymb}
 \usepackage{csquotes}
  \usepackage{datetime}
\usepackage{amsmath}
\numberwithin{equation}{subsection}
\usepackage{amsthm}
\usepackage{mathrsfs}
\usepackage[dvips]{graphicx}
\usepackage{vmargin}
\usepackage{color}
\usepackage{amscd}
\usepackage{wasysym}
\usepackage{float}
\usepackage{stmaryrd} 
\DeclareGraphicsExtensions{.pdf, .png, .jpg, .mps}
\usepackage[colorlinks=true,pdfpagemode=FullScreen]{hyperref}
\definecolor{myc}{cmyk}{0.0009,0.8,0.8,0.00}
\newtheorem{theorem}{Theorem}[section]
\newtheorem{lem}[theorem]{Lemma}
\newtheorem{pro}[theorem]{Proposition}
\newtheorem{defi}[theorem]{Definition}

\newtheorem{rem}[theorem]{Remark}

\newtheorem{notation}[theorem]{Notation}
\newtheorem{assumption}[theorem]{Assumption}
\newtheorem*{nb}{\footnotesize {N.B}}

\def\p{\partial}
\def\no{\noindent}
\def\io{{\infty}}

\def\re{\operatorname{Re}}
\def\card{\operatorname{Card}}
\def\im{\operatorname{Im}}
\def\Id{\operatorname{Id}}
\def\Lg{\operatorname{Log}}

\def\N{\mathbb N}
\def\Z{\mathbb Z}

\def\R{\mathbb R}
\def\C{\mathbb C}
\def\poscal#1#2{\langle#1,#2\rangle}

\def\norm#1{\Vert#1\Vert}

\def\val#1{\vert#1\vert}
\def\Val#1{\left\vert#1\right\vert}

\def\l2{L^2(\R^{n})}
\def\L2{L^2(\R^{2n})}
\def\supp{\operatorname{supp}}

\def\sign{\operatorname{sign}}

\def\RZ{\R^{2n}}

\def\sign{\operatorname{sign}}

\def\hs{{\hskip15pt}}
\def\vs{\vskip.3cm}
\def\indic#1{\mathbf 1_{#1}}
\let \dis=\displaystyle
\let\no=\noindent
\let \dis=\displaystyle

\let\no=\noindent
\def\mat22#1#2#3#4{\begin{pmatrix}#1&#2\\ #3&#4\end{pmatrix}}

\def\XXint#1#2#3{{\setbox0=\hbox{$#1{#2#3}{\int}$}
     \vcenter{\hbox{$#2#3$}}\kern-.5\wd0}}

\def\beq{\begin{equation}}
\def\eeq{\end{equation}}

\def\dsp{\displaystyle}
\def\binome#1#2{{\scriptscriptstyle{\left( \begin{matrix} {\scriptstyle{#1}} \\[-0.5ex]
{\scriptstyle {#2}}  \end{matrix} \right)}}}
\def\tN{\texttt{N}}
\begin{document}
\baselineskip=1.1\normalbaselineskip
\title{On integrals over a  convex set of the Wigner distribution}
\author{B\'erang\`ere Delourme, Thomas Duyckaerts, Nicolas Lerner}
\address{\noindent \textsc{B. Delourme, LAGA, UMR 7539, Institut Galil\'ee, 
Universit\'e Paris 13,
99, avenue Jean-Baptiste Cl\'ement,
93430 - Villetaneuse
France}}
\email{delourme@math.univ-paris13.fr}
\address{\noindent \textsc{T. Duyckaerts, 
Institut Universitaire de France \&
LAGA, UMR 7539, Institut Galil\'ee, 
Universit\'e Paris 13,
99, avenue Jean-Baptiste Cl\'ement,
93430 - Villetaneuse
France}}
\email{duyckaer@math.univ-paris13.fr}
\address{\noindent \textsc{N. Lerner, Institut de Math\'ematiques de Jussieu,
Sorbonne Universit\'e (former Paris VI),
Campus Pierre et Marie Curie,
4 Place Jussieu,
75252 Paris cedex 05,
France}}
\email{nicolas.lerner@imj-prg.fr}
\numberwithin{equation}{subsection}
\begin{abstract}
We provide an example of a normalized $L^{2}(\R)$ function $u$
such that its Wigner distribution $\mathcal W(u,u)$ has an integral $>1$ on the square $[0,a]\times[0,a]$ for a suitable choice of $a$.
This provides a negative answer to a question raised by P. Flandrin in \cite{conjecture}.
Our arguments are based upon the study of the Weyl quantization of the indicatrix of ${\R_{+}\times\R_{+}}$
along with a precise numerical analysis of its discretization.
\end{abstract}
\maketitle
\centerline{\color{magenta}\boxed{\text{\bf \today}}}
\vfill
{\footnotesize
\color{red}\tableofcontents}
\vfill\eject
\section{Introduction}
\subsection{Flandrin's conjecture}\label{sec11}
Let $C$ be a convex bounded subset of $\R^{2n}$
 and $\mathbf 1_{C}$ be the characteristic function of $C$.
 A statement known as Flandrin's conjecture\footnote{On page 2178 of \cite{conjecture},
P.~Flandrin writes {\it ``it is conjectured that the result \eqref{flandrin}
is true for any convex domain $C$''},
a quite mild commitment for the validity of \eqref{flandrin},
although that statement was referred to later on as {\it Flandrin's conjecture}
in the literature.} asserts that for $u\in L^{2}(\R^{n})$, 
 \begin{equation}\label{flandrin}\dis 
\iint_{C} \mathcal W(u,u)(x,\xi) dx d\xi\le \norm{u}_{L^{2}(\R^{n})}^{2},
\end{equation}
where
the Wigner function $\mathcal W(u,v)$
is given by
\begin{equation}\label{wigner}
\mathcal W(u,v)(x,\xi)=\int e^{-2i\pi z\cdot \xi} u(x+\frac z2) \bar v(x-\frac z2) dz.
\end{equation}
Note that $\mathcal W(u,v)$ appears as the partial Fourier transform\footnote{For $f\in \mathscr S(\R^{N})$,
we define its Fourier transform by 
$
\hat f(\xi)=\int_{\R^{N}} e^{-2i\pi x\cdot \xi} f(x) dx
$ and we obtain the inversion formula
$ f(x)=\int_{\R^{N}} e^{2i\pi x\cdot \xi} \hat f(\xi) d\xi$. Both formulas can be extended to tempered distributions.
} with respect to $z$ of the function 
$$
\R^{n}\times \R^{n}\ni(z,x)\mapsto u(x+\frac z2) \bar v(x-\frac z2) =\Omega(u,v)(x,z),
$$
and since for $u, v\in L^{2}(\R^{n})$, $\Omega(u,v)$ belongs to $L^{2}(\RZ)$ from the identity
$$
\int_{\RZ}\val{\Omega(u,v)(x,z)}^{2} dxdz=\norm{u}^{2}_{L^{2}(\R^{n})}\norm{v}^{2}_{L^{2}(\R^{n})},
$$
the function $\mathcal W(u,v)$ is in $L^{2}(\RZ)$ with 
\begin{equation}\label{norm}
\norm{\mathcal  W(u,v)}_{L^{2}(\RZ)}=\norm{u}_{L^{2}(\R^{n})}\norm{v}_{L^{2}(\R^{n})},
\end{equation}
giving a meaning to the integral in \eqref{flandrin} for $C$ with a finite Lebesgue measure.
We may note as well that although the function $\mathcal W(u,v)$ is complex-valued, the function $\mathcal  W(u,u)$ is in fact real-valued.
In fact we have readily
\begin{equation}\label{}
\overline{\mathcal W (u,v)(x,\xi)}=\mathcal W (v,u)(x,\xi).
\end{equation}
We note also that the real-valued function $\mathcal W (u,u)$ can take negative values,
choosing for instance $u_{1}(x)=x e^{-\pi x^{2}}$ on the real line, we get
$$
\mathcal W (u_{1},u_{1})(x,\xi)=2^{1/2} e^{-2\pi(x^{2}+\xi^{2})}\bigl(x^{2}+\xi^{2}-\frac1{4\pi}\bigr).
$$
\subsection{A reformulation of Flandrin's conjecture, state of the art, main result}
It is  easy to see that for $u,v$ in the Schwartz class $\mathscr S(\R^{n})$,
the remark above on the Fourier transform ensures that $\mathcal  W(u,v)$ belongs as well to $\mathscr S(\RZ)$ and we can reformulate \eqref{flandrin} as
\begin{equation}\label{flandrin++}
\forall C \text{convex bounded $\subset\RZ$,}
\forall u\in \mathscr S(\R^{n}), \iint_{C}W(u,u)(x,\xi) dx d\xi
\le\norm{u}_{L^{2}(\R^{n})}^{2}.
\end{equation}
The latter property follows from \eqref{flandrin} and conversely,
let us consider $u\in L^{2}(\R^{n})$, $\phi\in \mathscr S(\R^{n})$. We have 
$$
\mathcal W(u,u)=\mathcal W(u-\phi,u)+\mathcal W(\phi,u-\phi)+\mathcal W(\phi,\phi),
$$ 
so that for a subset $C$ of $\RZ$ with finite Lebesgue measure, we have, thanks to \eqref{norm},
$$
\poscal{\mathcal W(u,u)}{\indic{C}}_{L^{2}(\RZ)}\le 2\val{C}^{1/2}\norm{u-\phi}_{L^{2}(\R^{n})}
\norm{\phi}_{L^{2}(\R^{n})}+\poscal{\mathcal W(\phi,\phi)}{\indic{C}}_{L^{2}(\RZ)},
$$
so that if $C$ is convex bounded and \eqref{flandrin++} holds true, we have 
$$
\poscal{\mathcal W(u,u)}{\indic{C}}_{L^{2}(\RZ)}\le 2\val{C}^{1/2}\norm{u-\phi}_{L^{2}(\R^{n})}
\norm{\phi}_{L^{2}(\R^{n})}+\norm{\phi}_{L^{2}(\R^{n})}^{2}.
$$
Taking now $\phi$ as a sequence in the Schwartz space converging in $L^{2}(\R^{n})$ towards $u$,
we obtain
\eqref{flandrin++} for any $u\in L^{2}(\R^{n})$.
\begin{lem}
 Property \eqref{flandrin++} is equivalent to the same statement where the requirement $C$ bounded is removed.
\end{lem}
\begin{proof}
 When $C$ is convex with infinite Lebesgue measure, with $u\in \mathscr S(\R^{n})$,
we have $W(u,u)\in \mathscr S(\RZ)$ so that, thanks to the Lebesgue dominated convergence Theorem,
$$
\iint_{C}W(u,u)(x,\xi) dx d\xi=\lim_{\lambda\rightarrow+\io}\iint_{C\cap\{(x,\xi), \max(\val x, \val\xi)\le \lambda\}}W(u,u)(x,\xi) dx d\xi,
$$
and \eqref{flandrin++} implies
$
\iint_{C}W(u,u)(x,\xi) dx d\xi\le \norm{u}_{L^{2}(\R^{n})}^{2}.
$
\end{proof}
That property is easy for $C$ equal
to a half-space
$
C=\{(x,\xi)\in \RZ, L(x,\xi)\ge 0\},
$
where $L$ is a linear  form since
we can find (if $L\not=0$)
a unitary operator $\mathcal M$ on $L^{2}(\R^{n})$ such that, for $u\in \mathscr S(\R^{n})$,
$$
\iint_{\{(x,\xi), L(x,\xi)\ge 0\}} W(u,u)(x,\xi) dx d\xi=\iint_{y_{1}\ge 0} W(\mathcal M u,\mathcal M u)(y,\eta) dy d\eta
=\poscal{H_{1}\mathcal M u}{\mathcal M u},
$$ 
where $H_{1}$ is the operator of multiplication
by $H(y_{1})$, which is   an orthogonal  projection (thus has  norm 1):
this is a consequence of the symplectic covariance properties of the Wigner distribution
detailed in the next section,
but in that particular case, it is easy to choose linear symplectic coordinates
$y_{1},\dots, y_{n}, \eta_{1},\dots, \eta_{n}$ such that
$
y_{1}=L(x,\xi).
$
\par
Property \eqref{flandrin++}
is true as well for two-dimensional Euclidean disks and follows from a precise study of 
P. Flandrin (see e.g. \cite{MR1681043}): for $a\in \R_{+}$, defining
\begin{equation}\label{}
D_{a}=\{(x,\xi)\in \R^{2n}, \val x^{2}+\val {\xi}^{2}\le \frac a{2\pi}\},
\end{equation}
the paper  \cite{MR1681043} contains the proof of the estimate for $n=1$, 
\begin{equation}\label{134}
\iint_{D_{a}} W(u,u)(x,\xi) dx d\xi\le (1-e^{-a})\norm{u}_{L^{2}(\R)}^{2},
\end{equation}
for any $u\in L^{2}(\R)$.
The results for the disk in  two dimensions are readily extendable to polydisks by tensorisation.
A non-trivial matter was to extend this study
to $2n$-dimensional Euclidean balls, a task done in 
  the paper \cite{MR2761287} by E.~Lieb and Y.~Ostrover, who  provided the case where $C$ is chosen as an Euclidean ball. 
  As for  the argument of \cite{MR1681043},
  a highly non-trivial inequality on Laguerre polynomials 
 provides a proof of the estimate for $n\ge 1$, 
 \begin{equation}\label{}
\iint_{D_{a}} W(u,u)(x,\xi) dx d\xi\le
 \Bigl(1-\frac1{(n-1)!}\int_{a}^{+\io} e^{-t}t^{n-1}dt\Bigr)\norm{u}_{L^{2}(\R^{n})}^{2},
\end{equation}
for any $u\in L^{2}(\R^{n})$.
It turns out that the above short summary contains most of our knowledge on Flandrin's conjecture and for instance, the cases of the square  $C=[0,a]\times [0,a]$ or of $\ell^{p}$ balls of $\R^{2}$ ($p\not=2$)
were not explicitly explored in the literature.
Our main result in this paper is the following theorem.
\begin{theorem}\label{thmmain}
 There exist $a>0$ and $u\in \mathscr S(\R)$ such that
 \begin{equation}\label{}
\iint_{[0,a]^{2}}W(u,u)(x,\xi) dx d\xi>\norm{u}^{2}_{L^{2}(\R)},
\end{equation}
where the Wigner distribution $W(u,u)$ is defined by \eqref{wigner}.
\end{theorem}
This theorem is proven in Section \ref{section5}.
It turns out that most of the properties of the Wigner distribution are inherited
from its links with the Weyl quantization introduced by  \textsc{Hermann Weyl} in 1926 in the first edition of \cite{MR0450450} and our first remarks are devised to stress that link.
\subsection{Weyl quantization}
\begin{defi}
Let $a\in \mathscr S'(\RZ)$. We define the Weyl quantization $a^{w}$ of the Hamiltonian $a$, by the formula
\begin{equation}\label{weylq}
(a^{w}u)(x)=\iint e^{2i\pi (x-y)\cdot \xi} a(\frac{x+y}2, \xi) u(y) dy d\xi,
\end{equation}
to be understood weakly as 
\begin{equation}\label{eza654}
\poscal{a^{w}u}{\bar v}_{\mathscr S'(\R^{n}), \mathscr S(\R^{n})}=\poscal{a}{\mathcal W(u,v)}_{\mathscr S'(\RZ), \mathscr S(\RZ)}.
\end{equation}
\end{defi}
We note that the sesquilinear  mapping 
$$\mathscr S(\R^{n})\times \mathscr S(\R^{n})\ni(u,v)\mapsto \mathcal W(u,v) \in \mathscr S(\RZ),$$
is continuous so that the above bracket of duality $
\poscal{a}{\mathcal W(u,v)}_{\mathscr S'(\R^{2n}), \mathscr S(\R^{2n})}
$
makes sense.
We note as well that a temperate distribution $a\in\mathscr S'(\RZ)$ gets quantized
by a continuous  operator $a^{w}$ from 
$\mathscr S(\R^{n})$ 
into
$\mathscr S'(\R^{n})$.
This very general framework  is not really useful since we want to compose our operators $a^{w}b^{w}$.
A first step in this direction is to look for sufficient conditions
ensuring that the operator $a^{w}$ is bounded on $L^{2}(\R^{n})$. 
Moreover, for $a\in \mathscr S'(\RZ)$ and $b$ a polynomial in $\C[x,\xi]$,
we have the composition formula,
\begin{align}
a^{w}b^{w}&=(a\sharp b)^{w},\label{gfcd44}
\\
(a\sharp b)(x,\xi)&= \sum_{k\ge 0}\frac1{(4i\pi)^{k}}\sum_{\val \alpha+\val \beta=k}
\frac{(-1)^{\val \beta}}{\alpha!\beta!}(\p_{\xi}^{\alpha}\p_{x}^{\beta}a)(x,\xi)
(\p_{x}^{\alpha}\p_{\xi}^{\beta}b)(x,\xi),\label{gfcd44+}
\end{align}
which involves here a finite sum.
This follows from (2.1.26) in \cite{MR2599384}
where several generalizations can be found.
\begin{pro}
Let $a$ be a tempered distribution on $\RZ$. Then we have 
 \begin{equation}\label{norm01}
\norm{a^{w}}_{\mathcal B(L^{2}(\R^{n}))}\le \min\bigl(2^{n}\norm{a}_{L^{1}(\RZ)}, \norm{\hat a}_{L^{1}(\RZ)}\bigr).
\end{equation}
\end{pro}
\begin{proof}
In fact we have from \eqref{eza654}, $u,v\in \mathscr S(\R^{n})$, 
$$
\poscal{a^{w}u}{ v}_{L^{2}(\R^{n})}=\iiint a(x,\xi) u(2x-y)\bar v(y) e^{-4i\pi(x-y)\cdot \xi} 2^{n}dy dxd\xi,
$$
so that defining for $(x,\xi)\in \RZ$ the operator $\sigma_{x,\xi}$ by
\begin{equation}\label{phase}
(\sigma_{x,\xi} u)(y)=u(2x-y)e^{-4i\pi(x-y)\cdot \xi}, 
\end{equation}
we see that $\sigma_{x,\xi}$ (phase symmetry) is unitary and self-adjoint and 
\begin{equation}\label{12hh}
a^{w}=2^{n}\iint a(x,\xi) \sigma_{x,\xi} dxd\xi,
\end{equation}
proving the first estimate of the proposition. 
As a consequence of \eqref{12hh},
we obtain that 
\begin{equation}\label{}
\left(a^{w}\right)^{*}=\left(\overline{a}\right)^{w},\quad\text{so that for $a$ real-valued, $(a^{w})^{*}=a^{w}$.}
\end{equation}

To prove the second estimate, we introduce the so-called ambiguity function $\mathcal A(u,v)$ as the inverse Fourier transform of the Wigner function $\mathcal W(u,v)$, so that for $u,v$ in the Schwartz class, we have
$$
(\mathcal A(u,v))(\eta, y)=\iint \mathcal W(u,v)(x,\xi) e^{2i\pi(x\cdot \eta+\xi\cdot y)} dx d\xi,
$$
i.e.
$$
(\mathcal A(u,v))(\eta, y)=\int u(z+\frac{y}2)\bar v(z-\frac{y}2) e^{2i\pi z\cdot \eta} dz.
$$
Applying Plancherel formula on \eqref{eza654},
we get
$$
\poscal{a^{w}u}{v}_{L^{2}(\R^{n})}=\poscal{\hat a}{\mathcal A(u,v)}_{\mathscr S'(\RZ), \mathscr S(\RZ)}.
$$
We note that a consequence of \eqref{gfcd44+} is that for a linear form $L(x,\xi)$,
we have
$$
L\sharp L=L^{2},\quad \text{and more generally}\quad L^{\sharp N}=L^{N.}
$$
As a result, considering for $(y,\eta)\in \RZ$, the linear form $L_{\eta, y}$ defined by
$$
L_{\eta, y}(x,\xi)=x\cdot \eta+\xi\cdot y,
$$
we see that 
$$
\mathcal A(u,v)(\eta, y)=\poscal{\bigl\{e^{2i\pi (x\cdot \eta+\xi\cdot y)}\bigr\}^{\text{\rm Weyl}}u}{v}_{L^{2}(\R^{n})},
$$
and thus we get Weyl's original formula 
$$
a^{w}=\iint \hat a(\eta, y) e^{iL_{\eta,y}^{w}} dyd\eta,
$$
which implies the second estimate in the proposition.
\end{proof}
\vs
A particular case of Segal's formula (see e.g. Theorem 2.1.2 in \cite{MR2599384})
is  with $F$ standing for the Fourier transformation,
\begin{equation}\label{fs55}
F^{*}a^{w}F=a(\xi,-x)^{w}.
\end{equation}
We defined the canonical symplectic form $\sigma$
on $\R^{n}\times \R^{n}$ with
\begin{equation}\label{}
\bigl[(x,\xi),(y,\eta)\bigr]=\xi\cdot y-\eta\cdot x.
\end{equation}
The symplectic group $\text{\sl Sp}(n,\R)$ is the subgroup of $S\in \text{\sl Sl}(2n,\R)$
such that 
\begin{equation}\label{55vhqs}
\forall X, Y\in \RZ, \quad [S X, S Y]=[X,Y], \qquad\text{i.e.\quad} S^{*}\sigma S=\sigma,
\end{equation}
with
\begin{equation}\label{ffqq99}
\sigma=\mat22{0}{I_{n}}{-I_{n}}{0}.
\end{equation}
The symplectic group is generated by
\begin{align}
&\text{(i)\quad } (x,\xi)\mapsto (Tx, ^{t\!}T^{-1}\xi), \quad T\in \text{\sl Gl}(n,\R),\label{sym1}\\
&\text{(ii)\quad } (x_{k},\xi_{k})\mapsto (\xi_{k},-x_{k}), \quad \text{other coordinates unchanged,}\label{sym2}\\
&\text{(iii)\quad } (x,\xi)\mapsto (x, \xi+Qx), \quad Q\in \text{\sl Sym}(n,\R).\label{sym3}
\end{align}
Now for $S\in \text{\sl Sp}(n,\R)$, the operator
\begin{equation}\label{segal}
(a\circ S)^{w}=\mathcal M^{*}a^{w}\mathcal M,
\end{equation}
where $\mathcal M$ belongs to the metaplectic group, which is a group of unitary transformations of $L^{2}(\R^{n})$.
Let us describe the generators of the metaplectic group corres\-ponding to the symplectic transformations (i-iii) above.
The metaplectic  group is generated by
\begin{align}
&\text{(j)\quad } (\mathcal M u)(x)=\val{\det T}^{-1/2} u(T^{-1}x),\label{met1}\\
&\text{(jj)\quad } \text{partial Fourier transformation with respect to $x_{k}$,}\label{met2}\\
&\text{(jjj)\quad }\text{multiplication by  $e^{i\pi \poscal{Qx}{x}}$}.\label{met3}
\end{align}
We note also that for $Y=(y,\eta)\in \RZ$, the symmetry $S_{Y}$ is defined by
$S_{Y}(X)=2Y-X$ and is quantized by the phase symmetry $\sigma_{Y}$ as defined by \eqref{phase}
with the formula
$$
(a\circ S_{Y})^{w}=\sigma_{Y}^{*} a^{w}\sigma_{Y}=\sigma_{Y}
a^{w}\sigma_{Y}.
$$
Similarly, 
the translation $T_{Y}$ is defined on the phase space by $T_{Y}(X)=X+Y$ and is quantized by the {\it phase translation} $\tau_{Y}$,
\begin{equation}\label{}
(\tau_{(y,\eta)} u)(x)=u(x-y) e^{2i\pi(x-\frac y2)\cdot\eta},
\end{equation}
and we have
\begin{equation}\label{}
(a\circ T_{Y})^{w}=\tau_{Y}^{*} a^{w}\tau_{Y}=\tau_{-Y}
a^{w}\tau_{Y}.
\end{equation}
Note also that the covariance formula \eqref{segal} can be reformulated as the following property of the Wigner distribution,
$$
\mathcal W\left(\mathcal M u, \mathcal Mv\right)=\mathcal W(u,v)\circ S^{-1}.
$$
Since the metaplectic group is continuous from $\mathscr S(\R^{n})$ into itself,
Segal's Formula \eqref{54879}
is valid as well for $a\in \mathscr S'(\RZ)$.
We note also that 
$
Sp(1,\R)=Sl(2, \R).
$
\begin{lem}
 Flandrin's conjecture \eqref{flandrin++} is equivalent to
 \begin{equation}\label{}
\forall C \text{convex bounded $\subset\RZ$,}\quad \indic{C}^{w}\le \Id.
\end{equation}
\end{lem}
\begin{proof}
 Indeed, since $C$ has  finite Lebesgue measure, the first inequality in \eqref{norm01}
 ensures that $\indic{C}^{w}$ is a bounded operator in $L^{2}(\R^{n})$
 and, thanks to \eqref{eza654}, Property \eqref{flandrin++} reads
 $$
 \forall u\in \mathscr S(\R^{n}), \quad \poscal{\indic{C}^{w}u}{u}\le \norm{u}^{2}_{L^{2}(\R^{n})},
 $$
 which indeed means $\indic{C}^{w}\le \Id.$
 \end{proof}
\section{The quarter-plane, elementary observations}
\subsection{Definitions}
We have chosen to focus our attention
on a most simple-looking case, when $C$ is the ``quarter-plane''
\begin{equation}\label{quarter}
C_{0}=\{(x,\xi)\in \R^{2}, \ x\ge 0, \xi\ge 0\}.
\end{equation}
We study in this section the operator 
\begin{equation}\label{quarter-plane}
A=\left(H(x)H(\xi)\right)^{w},
\end{equation}
where $H=\mathbf 1_{\R_{+}}$, 
that is the Weyl quantization of the characteristic function of the  first quarter of the plane. 
The  Hardy operator $\mathcal H_{a}$ is defined as the operator with distribution-kernel
\begin{equation}\label{hardy}
\frac{H(x) H(y)}{\pi(x+y)},
\end{equation}
and  is bounded on $L^{2}(\R)$ with operator-norm equal to 1
(see e.g. Lemma 4.1.8 in  \cite{MR2599384}).
\begin{pro}\label{pro21}
 The operator $A$ given by \eqref{quarter-plane} is bounded self-adjoint on $L^{2}(\R)$.
 Flandrin's conjecture for the quarter-plane \eqref{quarter} is $A\le I$.
\end{pro}
\begin{proof}
Since the Weyl symbol of $A$ is real-valued, $A$ is formally self-adjoint and it is enough to prove that $A$ is bounded on $L^{2}(\R)$.
Let us start with recalling the classical formulas
\begin{equation}\label{}
\hat H(t)=\frac{\delta_{0}(t)}{2}+\frac{1}{2i \pi}\text{\rm pv}\left(\frac1t\right),
\quad
\widehat{\sign}=\frac{1}{i \pi}\text{\rm pv}\left(\frac1t\right),
\end{equation}
useful below.
For $\lambda>0$, we define
$
A_{\lambda}=\left(H(x)\indic{[0,\lambda]}(\xi)\right)^{w},
$
whose distribution-kernel is the $L^{\io}(\RZ)$ function
$$
k_{\lambda}(x,y)=H(x+y)e^{i\pi(x-y)\lambda}\frac{\sin(\pi(x-y)\lambda)}{\pi(x-y)}.
$$
We can thus consider
$$
k_{\lambda}(x,y)=\underbrace{H(x) H(y)k_{\lambda}(x,y)}_{
\substack{k_{0,\lambda}(x,y)}}+\bigl(H(-x) H(y)+H(x)H(-y)\bigr)k_{\lambda}(x,y),
$$
and the operator with distribution-kernel $k_{0,\lambda}$ is $H\bigl(\indic{[0,\lambda]}(\xi)\bigr)^{w}H$,
that is  $H\indic{[0,\lambda]}(D)H$, where $H$ stands for the operator of  multiplication
by the Heaviside function $H$. 
On the other hand, the operator with distribution kernel $k_{1,\lambda}=k_{\lambda}-k_{0,\lambda}$
is such that
$$
\val{k_{1,\lambda}(x,y)}\le \frac{H(-x) H(y)+H(x)H(-y)}{\pi\val{x-y}}
=\frac{H(-x) H(y)}{\pi(y-x)}+\frac{H(x)H(-y)}{\pi(x-y)}.
$$
Since the Hardy operator $\mathcal H_{a}$ with kernel \eqref{hardy}
has norm 1,
we obtain 
that, for $u,v\in \mathscr S(\R^{n})$, with $H=H(x), \check H=H(-x)$, 
\begin{multline*}
\Val{\iint H(x) \indic{[0,\lambda]}(\xi)W(u,v)(x,\xi) dx d\xi}\le \norm{H u}_{L^{2}(\R)} \norm{H v}_{L^{2}(\R)} \\
+ \norm{H u}_{L^{2}(\R)}  
\norm{\check H v}_{L^{2}(\R)}+ \norm{\check H u}_{L^{2}(\R)}  
\norm{H v}_{L^{2}(\R)},
\end{multline*}
so that 
\begin{multline*}
\Val{\iint H(x)H(\xi)W(u,v)(x,\xi) dx d\xi}=
\lim_{\lambda\rightarrow+\io}\Val{\iint H(x) \indic{[0,\lambda]}(\xi)W(u,v)(x,\xi) dx d\xi}
\\
\le \norm{H u}_{L^{2}(\R)} \norm{H v}_{L^{2}(\R)} 
+ \norm{H u}_{L^{2}(\R)}  
\norm{\check H v}_{L^{2}(\R)}
+ \norm{\check H u}_{L^{2}(\R)}  
\norm{H v}_{L^{2}(\R)},
\end{multline*}
proving the $L^{2}$-boundedness of the operator $A$.
\begin{rem}\label{rem.5487}{\rm
That cumbersome detour with the operator $A_{\lambda}$ is useful to ensure that the operator $A$ is indeed bounded on $L^{2}(\R)$.
The kernel $k$
of $A$ is a distribution of order 1 and the product $H(x) H(y) k(x,y)$ is not meaningful, even when $k$ is a Radon measure. However with the $L^{2}$-boundedness of $A$, the products of operators
$HAH$,
$\check H A H$, 
$HA\check H$, $\check H A\check H$
make sense and for instance we may approximate in the strong-operator-topology the operator 
$HAH$ by the operator 
$
\chi(\cdot/\varepsilon) A\chi(\cdot/\varepsilon),
$
where $\chi$ is a smooth function supported in $[1,+\io)$ and equal to $1$ on $[2,+\io)$.
We have indeed
$$
HAH=\bigl(H-\chi(\cdot/\varepsilon)\bigr) A H+\chi(\cdot/\varepsilon) A \bigl(H-\chi(\cdot/\varepsilon)\bigr) 
+\chi(\cdot/\varepsilon)A \chi(\cdot/\varepsilon),
$$
so that for $u\in L^{2}(\R)$,
$
HAH u=\lim_{\varepsilon\rightarrow 0_{+}}\chi(\cdot/\varepsilon)A \chi(\cdot/\varepsilon)u.
$
The operator with kernel
$$
H(x+y)\chi(x/\varepsilon)\chi(y/\varepsilon) \textrm{\rm pv}\frac{1}{i\pi(y-x)}=\chi(x/\varepsilon)\chi(y/\varepsilon)  \text{\rm pv}\frac{1}{i\pi(y-x)}
$$
converges strongly towards  the operator 
$H(\sign D) H$.
}
\end{rem}
We can provide a slightly better estimate than above.
The kernel of  the $L^{2}$-bounded operator $A$ is\footnote{Note that, for $T_{1}, T_{2}$ distributions on the real line,  there is no difficulty at multiplying  $T_{1}(x+y)$ by  $T_{2}(x-y)$: in fact we may define
$$
\poscal{T_{1}(x+y)T_{2}(x-y)}{\phi(x,y)}_{\mathscr D'(\R^{2}),\mathscr D(\R^{2})}=\frac12
\poscal{T_{1}(x_{1})\otimes T_{2}(x_{2})}{\phi\bigl(\frac{x_{1}+x_{2}}2,\frac{x_{1}-x_{2}}2\bigr)}_{\mathscr D'(\R^{2}),\mathscr D(\R^{2})}.
$$
It is also a general consequence of the location of the wave-front-set of $T_{1}(x+y)$
(a subset of the conormal bundle of the second diagonal $x+y=0$) and of $T_{2}(x-y)$
(a subset of the conormal bundle of the diagonal $x-y=0$): we have 
$$
\{(x,-x;\xi,\xi)\in \R^{2}\times \R^{2} \}\cap\{(x,x;\xi,-\xi)\in \R^{2}\times \R^{2} \}=\{(0,0;0,0)\}.
$$} 
\begin{equation}\label{}
H(x+y) \hat H(y-x)=\frac{H(x+y)}2\left(\delta_{0}(y-x)+\frac{1}{i\pi(y-x)}\right),
\end{equation}
and with $\check H(x)=H(-x)$,
this implies readily (see Remark \ref{rem.5487}) that 
$\check H A \check H=0$.
In fact, for $u\in L^{2}(\R)$ with $\supp u\subset (-\io, 0)$, we have for $\varepsilon>0$ small enough
and $\chi$ as in Remark \ref{rem.5487},
$$
\poscal{\check H A\check H u}{u}_{L^{2}(\R)}=\poscal{\check \chi(\cdot/\varepsilon) A\check \chi(\cdot/\varepsilon)  u}{u}_{L^{2}(\R)}=0,
$$
since 
$$
\underbrace{\chi(-x/\varepsilon)}_{\substack{\text{supported on }\\
x\le -\varepsilon}}H(x+y)\underbrace{\chi(-y/\varepsilon)}_{\substack{\text{supported on }\\
y\le -\varepsilon}}\equiv 0.
$$
On the other hand, for $u\in L^{2}(\R)$ with $\supp u\subset(0,+\io)$,
we have for $\varepsilon>0$ small enough
and $\chi$ as in Remark \ref{rem.5487},
$$
\poscal{H A Hu}{u}_{L^{2}(\R)}=\poscal{\chi(\cdot/\varepsilon) A\chi(\cdot/\varepsilon)  u}{u}_{L^{2}(\R)},
$$
and the kernel of $\chi(\cdot/\varepsilon) A\chi(\cdot/\varepsilon) $ is
\begin{multline*}
\chi(x/\varepsilon) \chi(y/\varepsilon)\frac12 \delta_{0}(x-y)+\chi(x/\varepsilon) \chi(y/\varepsilon)\frac{1}{2i\pi(y-x)}
\\
=\chi(x/\varepsilon) \chi(y/\varepsilon)\times\text{kernel of $\frac{\Id+\sign D}2$}
\end{multline*}
so that 
\begin{equation}\label{215+}
HAH=H \frac{I+\sign D}2 H=H H(D) H,
\end{equation}
where $H$ is the operator of multiplication by $H(x)$ and $H(D)$ is the Fourier multiplier
$$
\widehat{H(D) u}(\xi)=H(\xi) \hat u(\xi).
$$
We have thus, defining $2\re \check H A H=\check H A H+H A \check H$,
\begin{equation}\label{789}
A=H AH+2\re \check H A H=H H(D) H+2\re \check H A H.
\end{equation}
The kernel of the operator $\check H A H$
is
\begin{equation}\label{54879}
\omega(x,y)=\check H(x) H(y)\frac{H(x+y)}{2i\pi (y-x)},
\end{equation}
and the kernel of $2\re \check H A H$ is 
\begin{multline}\label{more}
\check H(x)H(x+y) \hat H(y-x) H(y)
+
\check H(y)H(x+y) \hat H(y-x) H(x)
\\=H(x+y)\frac{\bigl(\check H(x) H(y)+\check H(y) H(x)\bigr) }{2i\pi (y-x)}.
\end{multline}
Using again the above estimate on  the Hardy operator,
we obtain that 
\begin{multline}\label{665544}
\poscal{Au}{u}_{L^{2}(\R)}\\=\norm{H(D) H u}^{2}_{L^{2}(\R)}+\re \iint\frac{H(x+y)\check H(x) H(y)}{i\pi(y-x)} 
(Hu)(y)\overline{(\check H u)(x)} dy dx,
\end{multline}
so that 
$$
\val{\poscal{Au}{u}_{L^{2}(\R)}}
\le \norm{Hu }^{2}_{L^{2}(\R)}+\norm{Hu }_{L^{2}(\R)}\norm{\check Hu }_{L^{2}(\R)}\le \frac{1+\sqrt 2}{2}\norm{u}^{2}_{L^{2}(\R)},
$$
proving the proposition with the estimate $\norm{A}_{\mathcal B(L^{2}(\R))}\le \frac{1+\sqrt 2}{2}\approx 1.2071$,
a rather crude estimate that we shall improve below (we have used here that the two-variable quadratic form
$x^{2}+xy$ has eigenvalues $\frac{1\pm \sqrt 2}{2}$).
\end{proof}
\begin{pro}
Let $A$ be given by \eqref{quarter-plane}.
With $H$ (resp. $\check H$)
standing for the operator of multiplication by $H(x)$ (resp. $H(-x)$)
and $H(D)$
for the Fourier multiplier $H(\xi)$,
we have
\begin{equation}\label{}
A= HH(D) H+\Omega, 
\end{equation}
where $\Omega=2\re \check H A H$ is the (self-adjoint) operator with kernel
$$
H(x+y)\frac{\bigl(\check H(x) H(y)+\check H(y) H(x)\bigr) }{2i\pi (y-x)}.
$$
The operator-norm of $\Omega$ is smaller than 1.
\end{pro}
\begin{proof}
 The first statements are proven  in \eqref{215+}, \eqref{54879} and the norm-estimate follows from the fact that the Hardy operator (cf. \eqref{hardy}) has norm 1.
\end{proof}
\subsection{Elementary calculations}
\begin{lem}\label{fl+++}
 Flandrin's conjecture for the quarter-plane is equivalent to 
 \begin{equation}\label{444}
H\check H(D) H- H A\check H A H\ge 0.
\end{equation}
\end{lem}
\begin{proof}
We have  from \eqref{789},
\begin{multline}\label{111}
\norm{u}^{2}-\poscal{Au}{u}=\norm{\check H(D) Hu}^{2}
+\norm{\check H u}^{2}-
2\re\poscal{\check H A Hu}{\check Hu}
\\
=\norm{\check H(D) Hu}^{2}+
\norm{\check H u-\check HA Hu}^{2}
-\norm{\check H A Hu}^{2}.
\end{multline}
If $\norm{u}^{2}-\poscal{Au}{u}\ge 0$ for all $u$, we get from \eqref{111} that with a given $u_{+}$ in $L^{2}(\R)$ supported in $\R_{+}$, we may choose 
 $$
 u_{-}= \check H A H u_{+},\quad \text{so that $u_{-}\in L^{2}(\R), \supp u_{-}\subset \R_{-}$},
 $$
 and we get  with $u=u_{+}+u_{-}$, 
 $$
 0\le \norm{u}^{2}-\poscal{Au}{u}=
 \norm{\check H(D) Hu_{+}}^{2}
-\norm{\check H A Hu_{+}}^{2},
 $$
 which is indeed \eqref{444}. Conversely, assuming \eqref{444} on $L^{2}(\R_{+})$,
 we get from \eqref{111} that $\norm{u}^{2}-\poscal{Au}{u}\ge 0$.
\end{proof}
\begin{lem}\label{556699}
 The kernel of the operator $H A\check H A H$ is 
 \begin{equation}\label{380}
k(x,y)=\frac{H(x) H(y)}{4\pi^{2}(x+y)}\frac{\Lg\bigl(1+\frac{\val{y-x}}{x+y}\bigr)}{\frac{\val{y-x}}{x+y}},
\end{equation}
\begin{equation}\label{5gftt}\text{and we have \quad}
\norm{\check H A H}\le \frac{1}{2\sqrt{\pi}}.
\end{equation}
\end{lem}
\begin{proof}
 The kernel $k(x,y)$
 of the operator $H A\check H A H$ is 
 $$
 \int H(x)\overline{\omega(z,x)}\check H(z)\omega(z,y) H(y) dz,
 $$
 where $\omega$ is given by \eqref{54879}.
 We have thus
 \begin{align*}
4k(x,y)&=H(x) H(y)\int_{-\io}^{0}
 \frac{H(x+z)}{i\pi(z-x)}
 \frac{H(z+y)}{i\pi(y-z)} dz
 \\&=
 \frac{H(x) H(y)}{\pi^{2}}\int_{-\io}^{0}
 \frac{H(x+z)}{(z-x)}
 \frac{H(z+y)}{(z-y)} dz
 \\&=\frac{H(x) H(y)}{\pi^{2}}\int_{\max(-x,-y)}^{0}
 \frac{1}{(z-x)(z-y)}
dz
\\&=\frac{H(x) H(y)}{\pi^{2}}\int_{-\min(x,y)}^{0}
 \frac{1}{(z-x)(z-y)}
dz
\\&=\frac{H(x) H(y)}{\pi^{2}}\int^{\min(x,y)}_{0}
 \frac{1}{(z+x)(z+y)}
dz
\\&=\frac{H(x) H(y)}{\pi^{2}(y-x)}\int^{\min(x,y)}_{0}
\left( \frac{1}{(z+x)}- \frac{1}{(z+y)}\right)
dz,
\end{align*} 
implying that 
\begin{align*}
4k(x,y)&=\frac{H(x) H(y)}{\pi^{2}(y-x)}\left[
\Lg\frac{\val{z+x}}{\val{z+y}}
\right]_{z=0}^{z=\min(x,y)}
\\&
=\frac{H(x) H(y)}{\pi^{2}(y-x)}
\Lg\left(\frac{x+\min(x,y)}{y+\min(x,y)}\frac y x\right).
\end{align*}
If $0<x\le y$ we get
\begin{multline*}
4k(x,y)=\frac{H(x) H(y)}{\pi^{2}(y-x)}
\Lg\left(\frac{x+x}{y+x}\times \frac y x\right)=
\frac{H(x) H(y)}{\pi^{2}(y-x)}
\Lg\left(\frac{2y}{y+x} \right)
\\
=
\frac{H(x) H(y)}{\pi^{2}(y-x)}
\Lg\left(\frac{y+x+y-x}{y+x} \right)
=\frac{H(x) H(y)}{\pi^{2}(y-x)}
\Lg\left(1+\frac{y-x}{y+x} \right).
\end{multline*}
If $0<y\le x$ we get
\begin{multline*}
4 k(x,y)=\frac{H(x) H(y)}{\pi^{2}(y-x)}
\Lg\left(\frac{x+y}{y+y}\times \frac y x\right)=
\frac{H(x) H(y)}{\pi^{2}(y-x)}
\Lg\left(\frac{x+y}{2x} \right)
\\
=
\frac{H(x) H(y)}{\pi^{2}(x-y)}
\Lg\left(\frac{y+x+x-y}{y+x} \right)
=\frac{H(x) H(y)}{\pi^{2}(x-y)}
\Lg\left(1+\frac{x-y}{y+x} \right).
\end{multline*}
Eventually we find that
\begin{equation}\label{ker1}
k(x,y)=\frac{H(x) H(y)}{4\pi^{2}\val{x-y}}\Lg\left(1+\frac{\val{x-y}}{y+x} \right)=
\frac{H(x) H(y)}{4\pi^{2}(x+y)}\frac{\Lg\left(1+\frac{\val{x-y}}{y+x} \right)}{\frac{\val{x-y}}{x+y}},
\end{equation}
which is \eqref{380}.
We note now that the smooth function given by 
$$
\psi(t)=\frac{\Lg(1+t)}{t},
$$
is decreasing\footnote{For $t>-1$, 
we have $\psi(t)=\int_{0}^{1}\frac{d\theta}{1+\theta t}
$
and 
$\frac{d}{d t}\{(1+\theta t)^{-1}\}=-(1+\theta t)^{-2} \theta\le 0$.
} on $(-1,+\io)$,  so that $\psi([0,1])=[\Lg 2, 1]$.
As a result, the kernel $k$ is symmetric non-negative and we have 
$$
k(x,y)\le \frac1 {4\pi}\frac{H(x) H(y)}{\pi(x+y)}.
$$
Since the Hardy operator with kernel $\frac{H(x) H(y)}{\pi(x+y)}$ has norm 1
(cf. Lemma 4.1.8 in  \cite{MR2599384}),
we obtain that the operator-norm of 
$$
H A\check H A H=(\check HA H)^{*}\ (\check H A H)
$$
is bounded above by $1/(4\pi)$, so that 
$
\norm{\check H A H}\le \frac{1}{2\sqrt \pi},
$
concluding the proof of the lemma.
\end{proof}
\subsection{An upper bound and non-positivity}
\begin{pro}\label{bound2}
 We have 
 \begin{equation}\label{best?}
\bigl(H(x) H(\xi)\bigr)^{w}\le\frac12+\frac{\sqrt{\pi+1}}{2\sqrt \pi}<
 1.0740884.
\end{equation}
\end{pro}
\begin{proof}
 We have  from \eqref{789} and \eqref{5gftt},
\begin{equation}\label{355}
\poscal{Au}{u}=\norm{H(D) Hu}^{2}+2\re\poscal{\check HA H Hu}{\check Hu}\le \norm{Hu}^{2}
+\frac{1}{\sqrt \pi}\norm{Hu}
\norm{\check H u}. 
\end{equation}
The eigenvalues of the quadratic form $x^{2}+\pi^{-1/2} xy$ in two dimensions are
$$
\frac12\pm\frac{\sqrt{\pi+1}}{2\sqrt \pi}, 
$$
entailing the result.
\end{proof}
\begin{pro}
 The operator $A$ is bounded self-adjoint on $L^{2}(\R)$ with norm less than 
 $\frac12+\frac{\sqrt{\pi+1}}{2\sqrt \pi}$. Moreover, the spectrum of $A$ intersects $(-\io,0)$ and the operator $A$ is not non-negative.
\end{pro}
\begin{proof}
 From \eqref{355}, we have
 $$
 \val{\poscal{Au}{u}}\le  \norm{Hu}^{2}
+\frac{1}{\sqrt \pi}\norm{Hu}
\norm{\check H u},
 $$
 and the first result on the norm follows from the reasoning in the proof of the previous proposition.
 The operator $A$ cannot be non-negative: if it were the case, we would have
 $$
 A=B^{2},\quad B=B^{*} \text{ bounded self-adjoint.}
 $$
 It would imply from \eqref{789}
 \begin{align*}
\poscal{Au}{u}&=\poscal{HAH u}{u}+2 \re \poscal{\check H A Hu}{\check H u}
 \\&=\poscal{HBBHu}{u}+
 2 \re \poscal{\check H BBHu}{\check H u}
 \\
 &=\norm{BH u}^{2}+2\re\poscal{BHu}{B\check H u}
 \\
 &=\norm{BH u+B\check H u}^{2}-\norm{B\check H u}^{2}
 \\
 &=\norm{Bu}^{2}-\norm{B\check H u}^{2}=\poscal{Au}{u}-\norm{B\check H u}^{2},
\end{align*} 
and thus 
$B\check H=0$, so that $\check H B=0$ and thus $\check H B^{2}=\check H A=0$,
so that $\check H AH=0$,
which is not true from Lemma \ref{556699}.
\end{proof}
\begin{lem}\label{lem.28}We have
$
A\ge -\frac1{\sqrt{\pi}}>-1.
$
\end{lem}
\begin{proof}
From \eqref{789} and \eqref{5gftt},
we have 
$$A=H AH+2\re \check H A H=H H(D) H+2\re \check H A H\ge 2\re \check H A H\ge-\frac{1}{\sqrt \pi},
$$
and thus we have 
\begin{equation}\label{moreb}
-1<-\frac{1}{\sqrt \pi}\le A\le \frac12+\frac{\sqrt{\pi+1}}{2\sqrt \pi}.\qedhere
\end{equation}
\end{proof}
\section{From the quarter-plane to an infinite matrix}
\subsection{Discretization}
With $A$ given by \eqref{quarter-plane} , $\varepsilon >0$, and $j,k\in \Z$,
we define
\begin{equation}\label{311}
a_{j,k,\varepsilon}=\poscal{A \mathbf 1_{(j\varepsilon, (j+1)\varepsilon)}}{ \mathbf 1_{(k\varepsilon, (k+1)\varepsilon)}}_{L^{2}(\R)},
\end{equation}
and for $l\in \Z$, we set
\begin{multline}\label{312}
\Phi(l)=\frac1{2\pi}\int_{l}^{{l+1}}\Lg{\left\vert\frac{y}{y-1}\right\vert} dy\\
=\frac1{2\pi}\bigl[y\Lg\val y-y-(y-1)\Lg \val{y-1}+(y-1)\bigr]_{l}^{l+1}\\=
\begin{cases}
 0\quad  \text{for $l=0$,}\\
 \frac1{2\pi}\bigl((l+1)\Lg (l+1)-2l \Lg l+(l-1)\Lg(l-1)\bigr)\quad  \text{for $l\ge 1$},\\
 -\Phi(-l) \quad  \text{for $l\le -1$.}
\end{cases}
\end{multline}
\begin{pro}\label{pro5555}
 We have $a_{j,k, \varepsilon}=\varepsilon a_{j,k, 1}$
 and using the notation $a_{j,k}=a_{j,k, 1}$,
 we have 
 \begin{multline}\label{2323}
a_{j,k}= \indic{\N}(j)\frac{\delta_{j,k}}2+i\Bigl(\indic{\N}(j+k)\Phi(k-j)+\frac{\delta_{j,-k-1}}2\Phi(2k+1)\Bigr)
\\
=\indic{\N}(j)\indic{\N}(k)\frac{\delta_{j,k}}2+i
\Phi(k-j)\Bigl(\indic{\N}(j+k)+\frac{\delta_{j+k+1, 0}}{2}     \Bigr).
\end{multline}
The operator  $\mathbf Q=(a_{j,k})$ is bounded self-adjoint on $\ell^{2}(\Z)$.
\end{pro}
\begin{nb} We have for $l\ge 1$
\begin{align*}
2\pi\Phi(l)&=\int_{l}^{{l+1}}\Lg{\left\vert\frac{y}{y-1}\right\vert} dy=\int_{0}^{1}\Lg (l+\theta) d\theta
-\int_{-1}^{0}\Lg(l+\theta) d\theta
\\
&=\int_{-1}^{1}\sign \theta\Lg(l+\theta) d\theta=\int_{0}^{1}\bigl(\Lg (l+\theta)-\Lg (l-\theta)\bigr) d\theta
\\
&=2\int_{0}^{1}\sum_{r\ge 1}\frac{(\theta/l)^{2r-1}}{2r-1} d\theta=\sum_{r\ge 1}\frac{1}{(2r-1) r l^{2r-1}},
\end{align*}
so that 
\begin{equation}\label{314}
\Phi(l)=\frac{1}{2\pi l}+\frac1{2\pi l^{3}}\sum_{r\ge 2}\frac{1}{(2r-1) r l^{2r-4}}, \quad l\ge 1.
\end{equation}
\end{nb}
\begin{proof}
Using Formula \eqref{segal}, with 
$$
(\mathcal M_{\varepsilon}u)(x)=
\varepsilon^{-1/2}u(x/\varepsilon)\quad
\text{and }\quad
S(x,\xi)=
(\varepsilon x, \varepsilon^{-1}\xi),
$$
 we have
 \begin{align*}
a_{j,k,\varepsilon}&=\varepsilon\poscal{\bigl(H(x) H(\xi)\bigr)^{w} \varepsilon^{-1/2} \mathbf 1_{[j,j+1]}(x/\varepsilon)}{ \varepsilon^{-1/2}\mathbf 1_{[k,k+1]}(x/\varepsilon)}
\\&=
\varepsilon\poscal{\bigl(H(x) H(\xi)\bigr)^{w}\mathcal M_{\varepsilon} \mathbf 1_{[j,j+1]}}{ \mathcal M_{\varepsilon}\mathbf 1_{[k,k+1]}}
\\&=
\varepsilon\poscal{\mathcal M_{\varepsilon}^{*}\bigl(H(x) H(\xi)\bigr)^{w}\mathcal M_{\varepsilon} \mathbf 1_{[j,j+1]}}{ \mathcal \mathbf 1_{[k,k+1]}}
\\&=
\varepsilon\poscal{
\bigl(H( \varepsilon x) H(\varepsilon^{-1}\xi)\bigr)^{w} \mathbf 1_{[j,j+1]}}{ \mathcal \mathbf 1_{[k,k+1]}}
  \\
 & =\varepsilon a_{j,k,1},
\end{align*} 
and with  $\mathbb I_{j}=\mathbf 1_{[j,j+1]}$, 
from \eqref{789}, \eqref{more}, we obtain
\begin{align}
a_{j,k}=\poscal{
H(D) H\mathbb I_{j}
}{H\mathbb I_{k}}
&+
\iint \frac{H(x+y)\bigl(\check H(x) H(y)+\check H(y) H(x)\bigr)}{2i\pi(y-x)} \mathbb I_{j}(y)\mathbb I_{k}(x)dxdy
\notag\\&=\underbrace{\indic{\N}(j)\indic{\N}(k)\poscal{
H(D) \mathbb I_{j}
}{\mathbb I_{k}}}_{c_{j,k}}
+b_{j,k}-b_{k,j},
\label{315}
\end{align} 
with
\begin{equation}\label{keyff2}
b_{j,k}=
\iint
\frac{H(x+y)\check H(x) H(y)}{2i\pi(y-x)} \mathbb I_{j}(y)\mathbb I_{k}(x)dxdy.
\end{equation}
On the other hand, we have 
\begin{equation}\label{}
\widehat{\mathbb I}_{j}(\xi)= \frac{\sin (\pi \xi)}{\pi \xi} e^{-i\pi(2j+1)\xi},\quad H(\xi)=\frac{1+\sign \xi}{2},
\end{equation}
and thus
\begin{align}
c_{j,k}&=\frac12\delta_{j,k}\indic{\N}(j)\indic{\N}(k)+\frac12\poscal{(\sign D)\mathbb I_{j}}{\mathbb I_{k}}\indic{\N}(j)\indic{\N}(k)
\notag\\
&=\frac12\delta_{j,k}\indic{\N}(j)\indic{\N}(k)+\frac12\indic{\N}(j)\indic{\N}(k)(1-\delta_{j,k})\poscal{(\sign D)\mathbb I_{j}}{\mathbb I_{k}}
\notag\\
&=\frac12\delta_{j,k}\indic{\N}(j)\indic{\N}(k)+\indic{\N}(j)\indic{\N}(k)(1-\delta_{j,k})\frac1{2i\pi}
\iint \frac{1}{(y-x)}
\mathbb I_{j}(y)\mathbb I_{k}(x) dy dx,
\label{mmm}\end{align}
where we note that in the above integral, for $j\ge k+1$
(resp. $k\ge j+1$) we have 
$$
j\le y\le j+1,\quad k\le x\le k+1\Longrightarrow 
\begin{cases}
 y-x\ge j-k-1\ge 0,&\text{ if $j\ge k+1$},
 \\
x-y\ge k-j-1\ge 0,&\text{ if $k\ge j+1$},
\end{cases}
$$
so that the integrand is non-negative (resp. non-positive).
We have also from \eqref{keyff2}
\begin{equation}\label{}
b_{j,k}=\indic{\N}(j)\indic{\N}(-k-1)
\iint
\frac{H(y-x)}{2i\pi(y+x)} \mathbb I_{j}(y)\mathbb I_{k}(-x)dxdy,
\end{equation}
so that finding $y>x$ in the support of the integrand implies that 
$$
-k-1\le x<y\le j+1\Longrightarrow j+k+2>0\quad\text{i.e.\quad} j+k+1\ge 0,
$$
so that $b_{j,k}$ is supported where $j+k+1\ge 0$.
\eject\no
$\bullet$
As a result, if $j+k+1\le -1$,
we have from \eqref{315}, \eqref{mmm},
$$
a_{j,k}= c_{j,k}=0,
$$
proving \eqref{2323} for $j+k+1\le -1$.
\par\no
$\bullet$
Let us tackle now the case where $j+k=-1$.
We have
\begin{align*}
b_{j, -j-1}&=\indic{\N}(j)
\iint
\frac{H(y-x)}{2i\pi(y+x)} \mathbb I_{j}(y)\mathbb I_{j}(x)dxdy
\\&=\frac{\indic{\N}(j)}{2i\pi}
\int_{j}^{j+1}\int_{x}^{j+1}\frac{dy}{y+x}
dx
\\&=\frac{\indic{\N}(j)}{2i\pi}
\int_{j}^{j+1}[\Lg\val{y+x}]^{y=j+1}_{y=x}
dx
\\&=\frac{\indic{\N}(j)}{2i\pi}
\int_{j}^{j+1}\Lg\Val{\frac{j+1+x}{2x}}
dx,
\end{align*}
and we  find also that
\begin{align*}
\int_{j}^{j+1}\Lg\Val{\frac{j+1+x}{2x}}
&dx=-\Lg 2+\int_{j}^{j+1}\Lg\val{j+1+x} dx-\int_{j}^{j+1}\Lg\val{x} dx
\\
&=-\Lg 2-\int_{j}^{j+1}\Lg\val{x} dx+2\int_{j+\frac12}^{j+1}\Lg\val{2s} ds
\\
&=-\Lg 2-\int_{j}^{j+1}\Lg\val{x} dx+\Lg2+2\int_{j+\frac12}^{j+1}\Lg\val{s} ds
\\
&=\int_{j+\frac12}^{j+1}\Lg\val{x} dx-\int_{j}^{{j+\frac12}}\Lg\val{x} dx
\\
&=\int_{j}^{{j+\frac12}}\Lg\Val{\frac{x+\frac12}{x}} dx=\int_{j}^{{j+\frac12}}\Lg\Val{\frac{2x+1}{2x}} dx
\\
&=\int_{2j+1}^{2j+2}\Lg\Val{\frac{y}{y-1}} dx=2\pi\Phi(2j+1),
\end{align*}
so that 
$
b_{j, -j-1}=\frac{\indic{\N}(j)}{i}\Phi(2j+1),
$
and
\begin{equation*}
b_{j, -j-1}-b_{-j-1,j}=\frac{\indic{\N}(j)}{i}\frac{\Phi(2j+1)}{2}
-\frac{\indic{\N}(-j-1)}{i}\frac{\Phi(-2j-1)}{2},
\end{equation*}
as well as from \eqref{315}, \eqref{mmm},
$$
a_{j,-j-1}=\frac{\Phi(2j+1)}{2i}\bigl(
\indic{\N}(j)+\indic{\N}(-1-j)
\bigr)=i\frac{\Phi(k-j)}{2},
$$
proving \eqref{2323} for $j+k+1=0$.
\par\no
\eject
$\bullet$
Assuming now
 $j+k\ge 0$, we have 
\begin{align*}
b_{j,k}&=\indic{\N}(j)\indic{\N}(-k-1)
\iint
\frac{H(y-x)}{2i\pi(y+x)} \underbrace{\mathbb I_{j}(y)\mathbb I_{k}(-x)}_{
x\le -k\le j\le y,\
\text{thus $y\ge x$}}
dxdy
\\
&=\indic{\N}(j)\indic{\N}(-k-1)
\iint
\frac{1}{2i\pi(y+x)}
\mathbb I_{0}(y-j)\mathbb I_{0}(-x-k)dxdy
\\
&=\indic{\N}(j)\indic{\N}(-k-1)\frac{1}{2i\pi}
\int_{-k-1}^{-k}
\bigl[\Lg\val{y+x}\bigr]^{y=j+1}_{y=j}
dx
\\
&=\indic{\N}(j)\indic{\N}(-k-1)\frac{1}{2i\pi}\int_{-k-1}^{-k}\Lg\Val{\frac{x+j+1}{x+j}} dx
\\
&=\indic{\N}(j)\indic{\N}(-k-1)\frac{1}{2i\pi}\int_{j-k}^{j+1-k}\Lg\Val{\frac{t}{t-1}} dt
\\
\text{\small (from \eqref{312})\ }&=\indic{\N}(j)\indic{\N}(-k-1)i \Phi(k-j).
\end{align*}
\vfill
$\bullet$ As a result for $j+k\ge 0$, we have from \eqref{315}, \eqref{mmm}, 
\begin{multline}\label{5588}
a_{j,k}=\frac12\delta_{j,k}\indic{\N}(j)\indic{\N}(k)+\indic{\N}(j)\indic{\N}(k)(1-\delta_{j,k})\frac1{2i\pi}
\iint \frac{1}{(y-x)}
\mathbb I_{j}(y)\mathbb I_{k}(x) dy dx
\\
\underbrace{+\indic{\N}(j)\indic{\N}(-k-1)i \Phi(k-j)
-\indic{\N}(k)\indic{\N}(-j-1)i \Phi(j-k)}_{i \Phi(k-j)
(\indic{\N}(j)\indic{\N}(-k-1)+
\indic{\N}(k)\indic{\N}(-j-1)
)
},
\end{multline}
and for $j=k\ge 0$,
we get
$
a_{j,j}=\frac12,
$
proving \eqref{2323} for $j=k\ge 0$.
\par\no
\vfill
$\bullet$
We are left with the case $j+k\ge 0, j\not=k$:
by symmetry, we may also assume $j<k$ so that $2k >0$ and thus $k\ge 1$.
We have in that case from \eqref{5588},
\begin{multline}\label{88hhdd}
a_{j,k}=\indic{\N}(j)\indic{\N}(k)\frac{i}{2\pi}
\iint \frac{1}{(x-y)}
\mathbb I_{j}(y)\mathbb I_{k}(x) dy dx
\\+i\Phi(k-j)\indic{\N}(k)\indic{\N}(-j-1).
\end{multline}
We have for $j\le y\le  j+1\le 0\le k\le x\le  k+1$ that $x-y\ge 0$ and the Fubini-Tonelli Theorem gives
\begin{align*}
\iint \frac{1}{(x-y)}
\mathbb I_{j}(y)\mathbb I_{k}(x)& dy dx=\int_{j}^{j+1}\left(\int_{k}^{k+1}\frac{dx}{x-y}\right) dy
\\
&=\int_{j}^{j+1}\bigl[\Lg\val{x-y}\bigr]^{x=k+1}_{x=k} dy
=-\int_{j}^{j+1}\Lg\Val{\frac{y-k}{y-k-1}} dy
\\
&=-\int_{j-k}^{j+1-k}\Lg\Val{\frac{t}{t-1}} dy=-2\pi \Phi(j-k)=2\pi \Phi(k-j),
\end{align*}
so that \eqref{88hhdd} gives for $k\ge 1$
$$
a_{j,k}=\indic{\N}(j)\indic{\N}(k)i\Phi(k-j)+i\Phi(k-j)\indic{\N}(k)\indic{\N}(-j-1)
=i\Phi(k-j),
$$
proving \eqref{2323} in the case $j+k\ge 0,j<k$.
In the symmetric case $j+k\ge 0,j>k$,
we use that (cf. \eqref{311}, \eqref{312})
$$
a_{j,k}=\overline{a_{k,j}}=-i\Phi(j-k)=i\Phi(k-j),
$$
completing the proof of \eqref{2323}.
\par
The matrix $\mathbf Q=(a_{j,k})$ is Hermitian since $A$ is self-adjoint:
 the real part is indeed symmetric (even diagonal) with 
\begin{equation}\label{}
 \re \mathbf Q=\frac12 \mathbb P_{+}, \quad \text{$\mathbb P_{+}$ projection on $\ell^{2}(\Z_{+}=\N)$},
\end{equation}
and the imaginary part of $\mathbf Q$ is skew-symmetric from \eqref{2323}
since $\Phi$ is odd,
with 
\begin{equation}\label{}
 \im \mathbf Q=\left(\Phi(k-j)\bigl(\indic{\N}(j+k)+\frac12\delta_{j+k+1,0}\bigr)\right)_{(j,k)\in\Z^{2}}.
\end{equation}
For $z\in \ell ^{2}(\Z)$, we have
$$
\poscal{\mathbf Q z}{z}=\sum_{j,k\in \Z} z_{j}\bar z_{k} a_{j,k}=\poscal{A\sum_{j}z_{j}\mathbb I_{j}}{\sum_{k}z_{k}\mathbb I_{k}},
$$
and since 
$
\norm{\sum_{j}z_{j}\mathbb I_{j}}_{L^{2}(\R)}=\norm{z}_{\ell^{2}(\Z)}
$,
we get 
$
\val{\poscal{\mathbf Q z}{z}}\le \norm{A}_{\mathcal B(L^{2}(\R))}\norm{z}_{\ell^{2}(\Z)}^{2}
$
so that $\mathbf Q$ is self-adjoint bounded.
The proof of Proposition \ref{pro5555} is complete.
\end{proof}
\subsection{The matrix related to the quarter-plane}
\begin{lem}\label{lem32}
 Let $\varepsilon>0$ be given.
 We define the set
$$
\mathcal V=\bigcup_{\substack{\varepsilon>0\\z\in \ell^{2}(\Z)\\\textrm{\rm with finite support}}}\bigl\{\sum_{j\in \Z} z_{j}\mathbf 1_{[j\varepsilon, (j+1)\varepsilon]}\bigr\}.
$$
The set $\mathcal V$ is dense in $L^{2}(\R)$.
\end{lem}
\begin{proof}
For $\varepsilon>0, z\in \ell^{2}(\Z)$, we set
$$
u_{\varepsilon, z}(x)=
\sum_{j\in \Z} z_{j}\mathbf 1_{[j\varepsilon, (j+1)\varepsilon]}(x),
\quad\text{so that \quad}
\norm{u_{\varepsilon, z}}^{2}_{L^{2}(\R)}=\varepsilon\sum_{j\in \Z}\val{z_{j}}^{2}.
$$
Let $\phi\in C^{0}_{c}(\R)$. Setting for $\varepsilon>0$,
$$d(x)=
\phi(x)-\sum_{j\in \Z}\phi (\varepsilon j)\mathbf 1_{[j\varepsilon, (j+1)\varepsilon]}(x),
$$
we find that, assuming 
$$
\supp \phi\subset[-R,R]\subset (-M\varepsilon, M\varepsilon),\quad M=1+\bigl[R/\varepsilon\bigr]
$$ 
with $\omega_{\phi}$ a modulus of continuity of $\phi$,
\begin{multline*}
\norm{d}_{L^{2}(\R)}^{2}=
\sum_{M\ge \val j}\int_{[j\varepsilon, (j+1)\varepsilon]}
\val{\phi(x)-\phi(\varepsilon j)}^{2}dx
\\
 \le \omega_{\phi}(\varepsilon)^{2}
(2M+1) \varepsilon
\le 
\bigl(2R+3\varepsilon\bigr) \omega_{\phi}(\varepsilon)^{2}.
\end{multline*}
As a consequence, the $L^{2}$ norm of $d$ can be made arbitrarily small, choosing $\varepsilon$ small enough.
\end{proof}
\begin{lem}\label{lem33}
 The matrix $\mathbf Q$ defined in Proposition \ref{pro5555} is self-adjoint and bounded
 and
 \begin{align}\label{}
\sup_{\norm{u}_{L^{2}(\R)}=1}
\poscal{Au}{u}&=\sup_{F \text{ \rm finite }\subset\Z}\lambda_{\text{\rm max}}(\mathbb P_{F}\mathbf Q\mathbb P_{F}),
\\\inf_{\norm{u}_{L^{2}(\R)}=1}
\poscal{Au}{u}&=
\inf_{F \text{ \rm finite }\subset\Z}\lambda_{\text{\rm min}}(\mathbb P_{F}\mathbf Q\mathbb P_{F}),
\end{align}
where $\mathbb P_{F}$ stands for the orthogonal projection on $\text{\rm span}\bigl(\{e_{j}\}_{j\in F}\bigr) $
where $e_{j}=(\delta_{k,j})_{k\in \Z}$.
\end{lem}
\begin{notation}\label{not34}
For $F$ finite subset of $\Z$, we shall use the notation $\mathbf Q_{F}$ for the $\card F\times\card F$ matrix
$\mathbb P_{F}\mathbf Q\mathbb P_{F}$.
\end{notation}
\begin{proof}
 We have from Lemma \ref{lem32},
\begin{multline*}
\sup_{\norm{u}_{L^{2}(\R)}=1}
\poscal{Au}{u}=\sup_{\substack{Z\in \ell^{2}(\Z), \norm{Z}_{\ell^{2}}=1, \varepsilon>0\\\text{$Z$ with finite support}}}\poscal{A\sum_{j}Z_{j}\mathbf 1_{[\varepsilon j, \varepsilon(j+1)]}}{\sum_{k}Z_{k}\mathbf 1_{[\varepsilon k, \varepsilon(k+1)]}} \varepsilon^{-1}
\\=\sup_{\substack{\substack{\varepsilon>0, Z\in \ell^{2}(\Z)\\
\text{$Z$ with finite support}
}}
}\poscal{\varepsilon \mathbf Q Z}{Z}_{\ell^{2}(\Z)}\norm{(\varepsilon^{1/2}Z)}_{\ell^{2}(\Z)}^{-2}
\\=\sup_{1=\norm{(z)}_{ \ell^{2}(\Z)}}\poscal{\mathbf Q z}{z}=\lambda_{\text{max}}(\mathbf Q),
\end{multline*}
where $\lambda_{\text{max}}(\mathbf Q)$ is the supremum of the spectrum of $\mathbf Q$.
In particular, taking $z\in \ell^{2}(\Z)$ with norm 1 and supported on a finite set $F$, we get that 
\begin{equation}\label{}
\sup_{\norm{u}_{L^{2}(\R)}=1}
\poscal{Au}{u}=\sup_{F \text{ finite }\subset\Z}\lambda_{\text{max}}(\mathbb P_{F}\mathbf Q\mathbb P_{F}).
\end{equation}
The other result can be obtained by changing $A$ into $-A$ in the above argument.
\end{proof}
\subsection{A consequence of Flandrin's conjecture}
\begin{theorem}\label{thmkey}
Let $\mathbf Q$ be the matrix defined in Proposition \ref{pro5555}.
A consequence of Flandrin's conjecture is that for all $F$ finite subset of $\Z$, we have with the notations of Lemma \ref{lem33},
\begin{equation}\label{324}
\mathbb P_{F}\mathbf Q\mathbb P_{F}\le 1.
\end{equation}
\end{theorem}
\begin{rem}
Lemma \ref{lem33} is also proving that if  \eqref{324} holds true for all finite subsets $F$ of $\Z$, then $A\le I$.
\end{rem}
\begin{proof}
 With $A$ defined by \eqref{quarter-plane}, Proposition \ref{pro21} and Flandrin's conjecture imply that 
 $A\le I$, so that applying Lemma \ref{lem33}, we obtain \eqref{324}.
\end{proof}
In the next section, we shall use Theorem \ref{thmkey}
to disprove Flandrin's conjecture by finding some finite subset $F$
 of $\Z$ such that the largest eigenvalue of
$\mathbb P_{F}\mathbf Q\mathbb P_{F}$ is strictly larger than 1.
Our proof will rely on a careful numerical analysis 
of the finite matrix $\mathbb P_{F}\mathbf Q\mathbb P_{F}$ for a suitable choice of $F$.
\section{Numerics}
\subsection{Main result and methodology}
\begin{theorem}\label{thmnum}
 There exists a finite subset $F$ of $\Z$ such that the self-adjoint matrix $\mathbf Q_{F}$ defined in Notation \ref{not34} has an eigenvalue strictly larger than 1.
\end{theorem}
This section is dedicated to the proof of Theorem~\ref{thmnum}. That proof  is based on the combination of a numerical computation with a rigorous numerical error analysis.  
We recall  that in any  computer,  numbers are stored with a limited precision (finite arithmetic precision, or floating-point arithmetic), see~\cite{goldberg1991,ascher2011first, demailly2016analyse}. That is why, for any number $a$, we  shall make the distinction between $a$ and its numerical representation that we denote by $a^{\texttt{N}}$.  In the sequel, we use the standard double precision accuracy, meaning that, for any real number $a$, the relative error between $a$ and its numerical representation is bounded above by $\varepsilon_r = 2^{-52}$ namely 
$$
\frac{| a - a^{\texttt{N}}| }{|a|} \le \varepsilon_r.
$$ 
The number $\varepsilon_r$ is often referred to as the machine precision. In the present case, since we deal with complex numbers, we shall introduce
\begin{equation}\label{definitionVarepsilon}
\varepsilon = \sqrt{2} \, (2 + \varepsilon_r) \, \varepsilon_r \leq 6.5 \times10^{-16}.
\end{equation} We also notice that  any  numerical computation (multiplication, addition or subtraction) introduces an additional round-off error (see Section~\ref{SectionSommeProduitNum})
which could be significant. This phenomenon explains why the use of finite arithmetic precision can lead to important numerical errors that we need to control.
\par
To prove Theorem~\ref{thmnum}, we choose $F= F_k \subset \Z$ defined by
\begin{equation}\label{DefinitionFk}
F_k = \{ i \in \Z, -k\leq i \leq k \}
\end{equation} 
and we take (somehow arbitrarily) $k=70$. The proof of Theorem~\ref{thmnum} is then divided into three main steps: 
\par\no
{\bf Step 1.} We compute numerically an eigenvector $x$ associated with the largest (in modulus) eigenvalue of $Q_{F}$.  We choose this eigenvector such that its numerical Euclidian norm is equal to $1$, that is to say 
$$
\bigl\vert (\| x \|_2^2 )^{\texttt{N}} -1\bigr\vert\le \varepsilon.
 $$
Then, we compute numerically the associated numerical Rayleigh quotient 
$$
\mathscr{R}^{\texttt{N}} = |{\poscal{Q_{F}^{\texttt{N}} x}{x}^{\texttt{N}}}|,
$$  
where $\poscal{\cdot} {\cdot}$ stands here for the Euclidian scalar product on $\C^{2k+1}$
and $Q_{F}^{\texttt{N}}$ denotes the numerical approximation of $Q_F$. We deliberately added  the superscript {\tt N} after each operation to remind that these operations are made numerically. 
We remark that 
\begin{equation}\label{EvaluationNumeriqueQuotientRayleigh}
\mathscr{R}^{\texttt{N}} \geq 1.00007. 
\end{equation} 
\par\no
{\bf Step 2.} Using a standard error analysis, we evaluate the numerical error made on the evaluation of the Rayleigh quotient. More specifically, we prove that
\begin{equation}\label{EvaluationErreurNumeriqueQuotientRayleigh}
\mathscr{E}^{\texttt{N}} = | \mathscr{R} - \mathscr{R}^{\texttt{N}} | \leq 10^{-9} \quad \mbox{with} \quad \mathscr{R} =  \left| \frac{ \poscal{Q_{F} x}{x}}{ \| x \|_2^{2}}   \right|.
\end{equation} 
{\bf Step 3.} To conclude the proof, we collect the results of~\eqref{EvaluationNumeriqueQuotientRayleigh} and \eqref{EvaluationErreurNumeriqueQuotientRayleigh} and we deduce the following inequality for the spectral radius $\rho(Q_{F})$  of $Q_F$: 
$$
\rho(Q_{F}) \geq   \mathscr{R}  =   \mathscr{R}^{\texttt{N}} +  ( \mathscr{R} - \mathscr{R}^{\texttt{N}} ) >   |\mathscr{R}^{\texttt{N}}| - |   \mathscr{R} - \mathscr{R}^{\texttt{N}} | > 1.
$$   
\par
The remainder of this section is organized as follows. In Section~\ref{SectionSommeProduitNum}, we briefly remind several classical results related to numerical errors arising from the standard arithmetic operations. Based upon these results, a first estimate of $\mathscr{E}^{\texttt{N}}$ is proven in Section~\ref{SectionComputationQxx} (Lemma~\ref{LemmeEvaluationR}). Finally, numerical computations, and, in particular Estimates~\eqref{EvaluationNumeriqueQuotientRayleigh}-\eqref{EvaluationErreurNumeriqueQuotientRayleigh} are given in Section~\ref{SubsectionResNum}.
\subsection
{General results on the round-off error resulting from numerical sums and products}
\label{SectionSommeProduitNum}
This section makes use of the basic results from~\cite[Chapter 1]{demailly2016analyse}, which  will be useful in Section~\ref{SectionComputationQxx}.
\subsubsection{Evaluation of the round-off error when adding two numbers}
\begin{lem}\label{lem41}
Let $a,b$ be  complex numbers and let us denote by $a^{\text{\tt N}}, b^{\text{\tt N}}$
some numerical approximations of $a,b$,
with
\begin{equation}
\begin{array}{lclll}
\dsp a^{\text{\tt N}}&=&a+\Delta a, & \val{\Delta a}\le\eta_{a}, & \eta_{a}>0,
\\[1ex]
\dsp b^{\text{\tt N}}&=&b+\Delta b, & \val{\Delta b}\le\eta_{b}, & \eta_{b}>0.
\end{array}
\end{equation}
Let $S^{\text{\tt N}}(\alpha,\beta)$ be a numerical approximation of $\alpha+\beta$ using finite precision arithmetic with a machine error $\varepsilon_r$: we have 
$$
\bigl\vert S^{\text{\tt N}}(a^{\text{\tt N}},b^{\text{\tt N}})-(a+b)\bigr\vert\le \varepsilon\bigl(\val a+\val b\bigr)+\eta_{a}(1+\varepsilon)
+\eta_{b}(1+\varepsilon),
$$ 
where $\varepsilon$ is defined in~\eqref{definitionVarepsilon}.
\end{lem}
\begin{proof}
First, let us consider the case where all the numbers $a$, $a^\tN$, $b$ and $b^\tN$ are real. One has
$$
\bigl\vert S_r^{\text{\tt N}}(a^{\text{\tt N}},b^{\text{\tt N}})-(a+b)\bigr\vert \leq \bigl\vert S_r^{\text{\tt N}}(a^{\text{\tt N}},b^{\text{\tt N}})-(a^{\text{\tt N}}+b^{\text{\tt N}})\bigr\vert+
\vert a^{\text{\tt N}}-a\vert+\vert b^{\text{\tt N}}-b\vert
$$   
Here, we use the subscript $r$ to remind that $S_r^\tN$ acts on real number. In view of ~\cite[Section 1.3]{demailly2016analyse}, we get
\begin{equation}\label{SommeDemailly}
\bigl\vert S_r^{\text{\tt N}}(a^{\text{\tt N}},b^{\text{\tt N}})-(a^{\text{\tt N}}+b^{\text{\tt N}})\bigr\vert \leq \varepsilon_r (\val{a^{\text{\tt N}}} +\val{b^{\text{\tt N}}}).
\end{equation}
As a result, applying the triangular inequality leads to 
\begin{equation}\label{SommeDeuxNombreReels}
\bigl\vert S_r^{\text{\tt N}}(a^{\text{\tt N}},b^{\text{\tt N}})-(a+b)\bigr\vert\le \varepsilon_r\bigl(\val a+\val b\bigr)+\eta_{a}(1+\varepsilon_r)
+\eta_{b}(1+\varepsilon_r).
\end{equation} 
In the complex case, we still have  
 \begin{equation}\label{EqSommeComplexe1}
\bigl\vert S^{\text{\tt N}}(a^{\text{\tt N}},b^{\text{\tt N}})-(a+b)\bigr\vert \le
 \bigl\vert S^{\text{\tt N}}(a^{\text{\tt N}},b^{\text{\tt N}})-(a^{\text{\tt N}}+b^{\text{\tt N}})\bigr\vert+
\vert a^{\text{\tt N}}-a\vert+\vert b^{\text{\tt N}}-b\vert.
\end{equation} 
We decompose the modulus in term of its real and imaginary part, namely
$$
\bigl\vert S^{\text{\tt N}}(a^{\text{\tt N}},b^{\text{\tt N}})-(a^{\text{\tt N}}+b^{\text{\tt N}})\bigr\vert^2 = \bigl \vert \re\bigl(S^{\text{\tt N}}(a^{\text{\tt N}},b^{\text{\tt N}})-(a^{\text{\tt N}}+b^{\text{\tt N}}))\bigr\vert^2 + 
\bigl \vert \im(S^{\text{\tt N}}(a^{\text{\tt N}},b^{\text{\tt N}})-(a^{\text{\tt N}}+b^{\text{\tt N}})\bigr)\bigr\vert^2, 
$$ 
and, setting $a^\tN = x_a^\tN + i y_a^\tN$ and $b^\tN = x_b^\tN + i y_b^\tN$, Formula~\eqref{SommeDemailly}  gives
\begin{align*}
&& \bigl \vert \re\bigl(S^{\text{\tt N}}(a^{\text{\tt N}},b^{\text{\tt N}})-(a^{\text{\tt N}}+b^{\text{\tt N}})\bigr)\bigr\vert^2 \leq \varepsilon_r^2 \left( | x_a^\tN|+    | x_b^\tN|\right)^2,\\ 
&& \bigl \vert \im\bigl(S^{\text{\tt N}}(a^{\text{\tt N}},b^{\text{\tt N}})-(a^{\text{\tt N}}+b^{\text{\tt N}})\bigr)\bigr\vert^2  \leq \varepsilon_r^2 \left( | y_a^\tN|+    | y_b^\tN|\right)^2.
\end{align*}
As a result, we have 
$$ \bigl\vert S^{\text{\tt N}}(a^{\text{\tt N}},b^{\text{\tt N}})-(a^{\text{\tt N}}+b^{\text{\tt N}})\bigr\vert  \leq \sqrt{2} \varepsilon_r ( |a^\tN| + |b^\tN|).
$$
Inserting the previous inequality into~\eqref{EqSommeComplexe1} and applying the triangle inequality we find that 
 \begin{equation}\label{EqSommeComplexe1+}
\bigl\vert S^{\text{\tt N}}(a^{\text{\tt N}},b^{\text{\tt N}})-(a+b)\bigr\vert \le \sqrt{2} \varepsilon_r ( \vert a \vert + \vert b \vert ) + \eta_a ( 1 + \sqrt{2} \varepsilon_r ) +  \eta_b ( 1 + \sqrt{2} \varepsilon_r ).
\end{equation}
Noticing that $ \sqrt{2} \varepsilon_r  < \varepsilon$ ends the proof.
\end{proof}
\subsubsection{Evaluation of the round-off error when adding successively $m$ numbers.}
 Let $m\in \N$, and assume that we are given  $m$ numerical approximations $(a_k^{\texttt{N}})_{k \in \llbracket 1, m\rrbracket}$ of some numbers  $(a_k)_{k \in \llbracket 1, m\rrbracket}$ such that
\begin{equation} \label{AssumptionSommeMultiple}
\forall k \in \llbracket 1, m \rrbracket,\quad  | a_k^{\texttt{N}} - a_k | \leq \eta_k. 
\end{equation}
We recall that numerically, the summation over $m$ terms is obtained by a recurrence procedure, which mimics that 
$$
\sum_{k=0}^m a_k = \Bigl( \sum_{k=0}^{m-1} a_k\Bigr) + a_m. 
$$  
That is why, for a given integer $j \geq 3$, we define by induction the numerical sum of $j$ terms $(a_k)_{k \in\llbracket 1, j\rrbracket }$ (using finite precision arithmetic with a machine error $\varepsilon$) as follows:
$$
S^{\texttt{N}}(a_{1}, a_{2}, \cdots, a_{j}) = S^{\texttt{N}} ( S^{\texttt{N}}(a_{1}, a_{2}, \cdots, a_{j-1}) , a_j).
$$ 
In view of the previous recurrence formula,  the order of summation of the terms may change the result of the numerical summation. 
For a given $m$, we shall evaluate  the error
$$
\Delta_m^S = \Bigl| S^{\texttt{N}}(a_{1}^{\texttt{N}}, a_{2}^{\texttt{N}}, \cdots, a_{m}^{\texttt{N}}) - \sum_{k=1}^m a_k \Bigr|. 
$$  
\begin{lem} Under Assumption~\eqref{AssumptionSommeMultiple},  for any $m\geq 2$,  we have
$$
 \Delta_{m+1}^S  \leq \varepsilon \Bigl( \sum_{k=1}^{m+1} |a_k| \Bigr) + \eta_{m+1} (1+\varepsilon) + \Delta_m^S (1+\varepsilon),
$$ 
and 
\begin{equation}\label{recurrenceSommen}
 \Delta_{m}^S  \leq \varepsilon \left( \sum_{q=0}^{m-2} (1+\varepsilon)^q \Bigl( \sum_{k=1}^{m-q} |a_k| \Bigr) \right) +  \sum_{q=1}^{m-1} (1+\varepsilon)^q \eta_{m-q+1} + \eta_1 (1+\varepsilon)^{m-1} .
\end{equation} 
\end{lem} 
\begin{proof}
\noindent A direct application of Lemma~\ref{lem41} gives
\begin{multline*}
 \Delta_{m+1}^S   \leq    \varepsilon \Bigl( |\sum_{k=1}^{m} a_k | + |a_{m+1}| \Bigr) +  \Delta_{m}^S (1+\varepsilon) + \eta_{m+1} (1+\varepsilon) \\ \leq
  \varepsilon \Bigl( \sum_{k=1}^{m+1} |a_k| \Bigr) +  \Delta_{m}^S (1+\varepsilon) + \eta_{m+1} (1+\varepsilon).
\end{multline*}
We prove the second estimate  by induction, the initialization step for $m=2$ being a direct consequence of Lemma~\ref{lem41}. Then, assuming that  \eqref{recurrenceSommen} holds for a given $m\geq 2$, we get 
\begin{eqnarray*}
\Delta_{m+1}^S  & \leq  & \varepsilon \Bigl( \sum_{k=1}^{m+1} |a_k| \Bigr) + \eta_{m+1} (1+\varepsilon) + \Delta_m^S (1+\varepsilon).
\end{eqnarray*}
Using the induction hypothesis, rearranging the  terms and applying the change of index $q +1\rightarrow q$ gives,
\begin{eqnarray*}
\Delta_{m+1}^S  & \leq &  \varepsilon \Bigl( \sum_{k=1}^{m+1} |a_k| \Bigr) + \varepsilon \Bigl( \sum_{q=0}^{m-2} (1+\varepsilon)^{q+1} \bigl( \sum_{k=1}^{m-q} |a_k| \bigr) \Bigr) \\
&&\hskip35pt +\sum_{q=1}^{m-1} (1+\varepsilon)^{q+1} \eta_{m-q+1}  +   \eta_{m+1} (1+\varepsilon)  +  \eta_1 (1+\varepsilon)^{m}  \\
& = &  \varepsilon \Bigl( \sum_{q=0}^{m-1} (1+\varepsilon)^{q} \bigl( \sum_{k=1}^{m+1-q} |a_k| \bigr) \Bigr)+    \eta_1 (1+\varepsilon)^{m}   +  \sum_{q=1}^{m} (1+\varepsilon)^{q} \eta_{m+1 -q+1}. 
\end{eqnarray*}  
\end{proof}
We notice that we can obtain from~\eqref{recurrenceSommen} the less accurate estimate for $\Delta_m^S$ given below.  Anyhow, this estimate turns out to be accurate  enough for the application that we have in mind. 
\begin{lem}\label{Somme_n_nombres}
Under Assumption~\eqref{AssumptionSommeMultiple}, for any $m\geq 2$, we have
\begin{equation}
\Delta_m^S  \leq \Bigl( \sum_{k=1}^{m-1} {\scriptstyle\binom{m-1}{k}}  \varepsilon^k \Bigr) \sum_{k=1}^{m} |a_k| +  (1 + \varepsilon)^{m-1} \sum_{q=1}^m \eta_q,
\end{equation}
where $\binom{m-1}{k}$ stands for the binomial coefficient.
\end{lem} 
\begin{proof}
It suffices to note that $\sum_{k=1}^{m-q} |a_k| \leq \sum_{k=1}^{m} |a_k|$, and consequently the first term in the right-hand-side of ~\eqref{recurrenceSommen} can be bounded above  by
\begin{multline*}
\varepsilon \left( \sum_{q=0}^{m-2} (1+\varepsilon)^q \Bigl( \sum_{k=1}^{m-q} |a_k| \Bigr) \right)  \leq   \varepsilon  \Bigl( \sum_{k=1}^{m} |a_k| \Bigr) \Bigl( \sum_{q=0}^{m-2} (1+\varepsilon)^q \Bigr) \\ =  \Bigl( \sum_{k=1}^{m} |a_k| \Bigr) \left( (1+\varepsilon)^{m-1} - 1\right) 
=  \Bigl( \sum_{k=1}^{m} |a_k| \Bigr) \sum_{k=1}^{m-1} \varepsilon^k{\scriptstyle\binom{m-1}{k}}.
\mbox{\qedhere}
\end{multline*}
\end{proof} 
\subsubsection{Evaluation of the round-off error when multiplying two numbers}
\begin{lem}\label{LemmeProduit}
Let $P^{\text{\tt N}}(\alpha, \beta)$ be a numerical approximation of the product $\alpha \beta$
using finite precision arithmetic with a machine error $\varepsilon_r$. Let $b^{\text{\tt N}}$ be a numerical approximation of a complex number $b$ such that 
$$
b^{\text{\tt N}}=b+\Delta b,\ \val{\Delta b}\le\eta_{b},\ \eta_{b}>0. 
$$  Then, we have for a complex number $a$,
 $$
 \val{P^{\text{\tt N}}(a,b^{\text{\tt N}})-ab}\le  \varepsilon\val a\val{b}+\eta_{b}\val a (1+\varepsilon).
 $$
\end{lem}
\begin{proof}
 First, we have 
  \begin{align}\label{ProduitComplexe1}
\val{P^{\text{\tt N}}(a,b^{\text{\tt N}})-ab}&\le \val{P^{\text{\tt N}}(a,b^{\text{\tt N}})-ab^{\text{\tt N}}}+\val{ab^{\text{\tt N}}-ab}
\end{align}
Here again, we separate the treatment of the real and the imaginary part:
\begin{equation}\label{SommeReelImag}
 \val{P^{\text{\tt N}}(a,b^{\text{\tt N}})-ab^{\text{\tt N}}}^2 = \val{ \re\left(P^{\text{\tt N}}(a,b^{\text{\tt N}})-ab^{\text{\tt N}}\right)}^2 + \val{ \im\left(P^{\text{\tt N}}(a,b^{\text{\tt N}})-ab^{\text{\tt N}}\right)}^2  
\end{equation}
Setting  $a = x_a + i y_a$ and $b^\tN = x_b^\tN + i y_b^\tN$,
 since $\re\left(a b^\tN\right) = x_a x_b^\tN - y_a y_b^\tN$, we have 
$$
\re\left(P^{\text{\tt N}}(a,b^{\text{\tt N}})\right) = S_r^\tN\left( P_r^\tN(x_a, x_b^\tN), -P_r^\tN(y_a, y_b^\tN)\right),
$$ 
where the subscript $r$ in $P_r^\tN$ and $S_r^\tN$  is used to highlight that these computations are made on real numbers.  We note that \cite[Section 1.4]{demailly2016analyse} gives
\begin{equation}\label{Inegalite1Produit}
\val{ P_r^\tN(x_a, x_b^\tN) - x_a x_b^\tN } \leq \varepsilon_r |x_a| |x_b^\tN| \quad \mbox{and} \quad \val{ P_r^\tN(y_a, y_b^\tN) - y_a y_b^\tN } \leq \varepsilon_r |y_a| |y_b^\tN|.
\end{equation}
As a result, applying formula~\eqref{SommeDeuxNombreReels} (specific to real numbers) leads to
\begin{equation}\label{Inegalite2Produit}
 \val{ \re\left(P^{\text{\tt N}}(a,b^{\text{\tt N}})-ab^{\text{\tt N}}\right)} \leq\varepsilon_r (2 + \varepsilon_r) \left( |x_a| |x_b^\tN| + |y_a| |y_b^\tN| \right).
 \end{equation}
 A similar analysis on the imaginary part gives
 \begin{equation}
 \val{ \im\left(P^{\text{\tt N}}(a,b^{\text{\tt N}})-ab^{\text{\tt N}}\right)} \leq \varepsilon_r (2 + \varepsilon_r) \left( |x_a| |y_b^\tN| + |y_a| |x_b^\tN| \right).
 \end{equation}
 Noticing that 
 $$
 \left( |x_a| |x_b^\tN| + |y_a| |y_b^\tN| \right)^2 +   \left( |x_a| |y_b^\tN| + |y_a| |x_b^\tN| \right)^2 \leq  2 |a|^2 |b^\tN|^2,
 $$ 
and introducing~\eqref{Inegalite1Produit}-\eqref{Inegalite2Produit}  into~\eqref{SommeReelImag}, we obtain
 $$
\val{P^{\text{\tt N}}(a,b^{\text{\tt N}})-ab^{\text{\tt N}}}^2  \leq 2 \varepsilon_r^2 (2 + \varepsilon_r)^2  |a|^2 |b^\tN|^2.
 $$ 
 Combining the previous inequality with~\eqref{ProduitComplexe1} and applying again the triangular inequality gives
 $$
 \val{P^{\text{\tt N}}(a,b^{\text{\tt N}})-ab} \leq  \sqrt{2} \varepsilon_r (2 + \varepsilon_r) | a| | b | + |a| \eta_b \bigl( 1 + \sqrt{2} \varepsilon_r (2 + \varepsilon_r) \bigr).
 $$  
 Substituting $\varepsilon$ for $\sqrt{2} \varepsilon_r (2 + \varepsilon_r)$  (see Definition~\eqref{definitionVarepsilon}) completes the proof.
\end{proof}
\subsection{Application to the numerical computation of  the Rayleigh quotient}\label{SectionComputationQxx}
In this section, we shall evaluate the numerical error made when evalua\-ting the Rayleigh quotient $\mathscr{R}^{\texttt{N}}$. The main result of this section is given by Lemma~\ref{LemmeEvaluationR}. It provides an estimate of $|\mathscr{R}^{\texttt{N}} - \mathscr{R}|$  under two appropriate assumptions (Assumption \ref{AssumptionDeltavarepsilon} and Assumption~\ref{assumtionDeltaEpsi}) listed below. Its proof requires to prove consecutively Lemma~\ref{LemmeEstimationProduitScalaire} and Lemma~\ref{LemmeEstimationProduitScalaire2}. 
\begin{assumption}\label{AssumptionDeltavarepsilon}
~\par\no 
$\mathbf{ (1)}$ The space $F$ is the space $F_k$ defined by \eqref{DefinitionFk}. We set $n = 2k+1$ so that $Q_F$ defined in Notation~\ref{not34} is a square matrix of size $n$ that satisfies
\begin{equation}\label{ConditionQij}
| (Q_F)_{ij} | < 1 \quad \forall (i, j) \,\in \llbracket 1,n \rrbracket^2.
\end{equation} 
$\mathbf{ (2)}$ There exists $\delta>0$ such that, for all $(i,j) \in \llbracket 1, n \rrbracket^2$,
\begin{equation}\label{erreurQij}
(Q_F)_{ij}^{\text{\tt N}} = (Q_F)_{ij} + \Delta (Q_F)_{ij} \quad \mbox{with }  |\Delta (Q_F)_{ij} | \leq \delta.
\end{equation} 
$\mathbf{ (3)}$  The vector $x$ is known explicitly  and satisfies
\begin{equation}\label{433}
\val{\left( \| x \|_2^2 \right)^{\text{\tt N}}-1} \le \varepsilon.
\end{equation}
\end{assumption}
Since we have,
$$
{\poscal{Q_{F} x}{x} }= \sum_{i=1}^n  \Bigl( \sum_{j=1}^n (Q_F)_{ij} x_j \Bigr) \overline{ x_i},
$$
we can decompose the numerical computation of ${\poscal{Q_{F}^{\texttt{N}} x}{x}^{\texttt{N}}}$ as follows:
we have 
$$
{\poscal{Q_{F}^{\texttt{N}} x}{x}^{\texttt{N}}} = S^{\texttt{N}}(a_1^{\texttt{N}}, a_2^{\texttt{N}}, \cdots, a_n^{\texttt{N}}),  
$$ 
where, for all $i \in \llbracket 1, n \rrbracket$,
\begin{equation}\label{DefinitionaiN}
a_i^{\texttt{N}} = P^{\texttt{N}}( S^{\texttt{N}}(p_{i,1}^{\texttt{N}}, p_{i,2}^{\texttt{N}}, \cdots, p_{i,n}^{\texttt{N}}), \overline{x_i}),
\end{equation} 
and, for all $(i,j) \in \llbracket 1, n \rrbracket^2$,
\begin{equation} \label{DefinitionpijN}
p_{i,j}^{\texttt{N}} = P^{\texttt{N}}( (Q_F)_{ij}^{\texttt{N}}, x_j).
\end{equation} 
Combining successively the results of Lemma~\ref{LemmeProduit} and  Lemma~\ref{Somme_n_nombres}, we can prove the following results.
\begin{lem}\label{LemmeEstimationProduitScalaire}Under Assumption~\ref{AssumptionDeltavarepsilon}, for any $(i,j) \in \llbracket 1, n \rrbracket^2$, we have 
$$
\left| \mathscr{R}^{\text{\tt N}}- \poscal{Q_{F} x}{x} \right| = \left| {\poscal{Q_{F}^{\text{\tt N}} x}{x}^{\text{\tt N}}} - \poscal{Q_{F} x}{x} \right| \leq  n  \|x\|_2^2  \beta(n,\delta, \varepsilon).
$$
Moreover,  we have 
$$
\beta(n,\delta, \varepsilon) = \left\{  \Bigl( \sum_{k=1}^{n-1} {\scriptstyle\binom{n-1}{k}} \varepsilon^k \Bigr) + (1 + \varepsilon)^{n-1} \bigl( \varepsilon +  \alpha(n, \delta,\varepsilon) (1+\varepsilon) \bigr) \right\}
$$
and
$$
\alpha(n,\delta, \varepsilon)  = \Bigl( \sum_{k=1}^{n-1}{\scriptstyle \binom{n-1}{k}} \varepsilon^k \Bigr)  +  (1 + \varepsilon)^{n-1} \bigl( \delta ( 1 + \varepsilon) + \varepsilon \bigr).
$$ 
\end{lem} 
\begin{proof}
\noindent First, a direct application of Lemma~\ref{LemmeProduit} gives that, for any $(i,j) \in \llbracket 1, n \rrbracket^2$, we have, using \eqref{ConditionQij},
\begin{multline}\label{Estimation1}
\eta_{i,j} = \left | p_{i,j}^{\texttt{N}} -(Q_F)_{ij}x_j  \right|= \left| P^{\texttt{N}}( (Q_F)_{ij}^{\texttt{N}}, x_j) -(Q_F)_{ij}x_j  \right| \\
\leq |x_j|  \bigl(\delta ( 1 + \varepsilon) + \varepsilon \bigr).
\end{multline} 
Then, Lemma~\ref{Somme_n_nombres} (taking $a_j=(Q_F)_{ij}x_j$)  along with~\eqref{ConditionQij} gives
\begin{align*}
&\left| S^{\texttt{N}}(p_{i,1}^{\texttt{N}}, p_{i,2}^{\texttt{N}}, \cdots ,p_{i,n}^{\texttt{N}}) - (Q_Fx)_i  \right|  
\\
&\hskip25pt\leq
 \Bigl( \sum_{k=1}^{n-1} \binome{n-1}{k} \varepsilon^k \Bigr) \sum_{k=1}^{n} |(Q_F)_{ik}x_k| +  (1 + \varepsilon)^{n-1} \sum_{q=1}^n \eta_{i,q} 
 \\
&\hskip25pt \leq  \Bigl( \sum_{k=1}^{n-1} \binome{n-1}{k} \varepsilon^k \Bigr) \| x \|_1 +  (1 + \varepsilon)^{n-1} \bigl( \delta( 1 + \varepsilon) + \varepsilon \bigr)  \| x \|_1 \leq  \| x \|_1 \alpha(n,\delta, \varepsilon).
\end{align*}
As a result, applying again Lemma~\ref{LemmeProduit}, we find that, for any $i \in \llbracket 1, n \rrbracket$,
\begin{eqnarray*}
\eta_i = | a_i^{\texttt{N}}  - (Q_Fx)_i \overline{x_i}  | & \leq  & \left| P^{\texttt{N}}( S^{\texttt{N}}(p_{i,1}^{\texttt{N}}, p_{i,2}^{\texttt{N}}, \cdots ,p_{i,n}^{\texttt{N}}) , \overline{x_i}) -(Q_Fx)_i \overline{x_i} \right| \\
 &\leq  &  \varepsilon \left| (Q_Fx)_i \right| |x_i|  +  \| x \|_1 \alpha(n,\delta, \varepsilon) |x_i| (1+\varepsilon)  \\ 
& \leq &  |x_i|  \| x \|_1 \bigl(\varepsilon +  \alpha(n, \delta,\varepsilon) (1+\varepsilon) \bigr).
\end{eqnarray*}
Here  we have used the fact that, thanks to~\eqref{ConditionQij}, $\left| (Q_Fx)_i \right| \leq \| x\|_1$. Finally, we apply Lemma~\ref{Somme_n_nombres} to obtain
\begin{align*}
&\left| {\poscal{Q_{F}^{\texttt{N}} x}{x}^{\texttt{N}}} - \poscal{Q_{F} x}{x} \right|  =   \left| S(a_1^{\texttt{N}}, a_2^{\texttt{N}}, \cdots ,a_n^{\texttt{N}})  -\poscal{Q_{F} x}{x} \right| \\
 \\&\hs\leq      \Bigl( \sum_{k=1}^{n-1}{\scriptstyle \binom{n-1}{k}} \varepsilon^k \Bigr) \sum_{k=1}^{n} |(Q_Fx)_k \overline{x_k}| +  (1 + \varepsilon)^{n-1} \sum_{q=1}^n  |x_q|  \| x \|_1 \bigl( \varepsilon +  \alpha(n,\delta ,\varepsilon) (1+\varepsilon) \bigr) 
 \\ 
&\hskip15pt \leq    \|x\|_1^2 \left\{  \Bigl( \sum_{k=1}^{n-1} {\scriptstyle\binom{n-1}{k}} \varepsilon^k \Bigr) + (1 + \varepsilon)^{n-1} \bigl( \varepsilon +  \alpha(n,\delta, \varepsilon) (1+\varepsilon) \bigr) \right\}. 
 \end{align*}
Using  that $\| x \|_1 \leq \sqrt{n} \| x \|_2$ ends the proof.
\end{proof}
To continue our computation, we shall make an additional assumption linking $\delta$ (error made on each coefficient the matrix $Q_F$) and the machine precision $\varepsilon$.
The previous assumption will be validated in Section~\ref{SubsectionResNum}.
\begin{assumption}\label{assumtionDeltaEpsi}
The positive numbers $\delta$ and $\varepsilon$ satisfy the following inequalities  
\begin{align} 
&0 < \varepsilon \ll \delta \leq 0.1,
  \\
   &n\varepsilon\leq \delta.
 \end{align} 
 \end{assumption}
\begin{lem}\label{LemmeEstimationProduitScalaire2}
We suppose  that Assumption~\ref{AssumptionDeltavarepsilon} and Assumption~\ref{assumtionDeltaEpsi} hold true. Then we have,
$$
\left| \mathscr{R}^{\text{\tt N}}- \poscal{Q_{F} x}{x} \right| = \left| {\poscal{Q_{F}^{\text{\tt N}} x}{x}^{\text{\tt N}}} - \poscal{Q_{F} x}{x} \right| \leq   7 \delta n \|x \|_2^2.
$$ 
\end{lem}
\begin{proof}
\noindent Under Assumption~\ref{assumtionDeltaEpsi}, we have 
$$
( 1 + \varepsilon)^n = \sum_{k=0}^{n}{\scriptstyle \binom{n}{k}} \varepsilon^k \leq \sum_{k=0}^{n} \frac{ (n\varepsilon)^k}{k!} \leq e^\delta, 
$$ 
which implies in particular that
\begin{equation}\label{439}
\sum_{k=1}^{n-1} {\scriptstyle\binom{n-1}{k}} \varepsilon^k   = ( 1 + \varepsilon)^{n-1} - 1 \leq e^\delta -1. 
\end{equation}
As a result, there exists $\xi_1  \in (0, \delta)$ such that
$$
\alpha(n,\delta,\varepsilon) \leq (e^\delta  - 1 + e^\delta\delta + e^\delta\varepsilon) \leq   e^\delta  - 1 + 2\delta e^\delta \leq 3 \delta + \delta^2 \bigl( 2 + e^{\xi_1} (\frac{1}{2} + \delta) \bigr).
$$ 
Since $\delta < 10^{-1}$,  $\delta^2 \left( 2 + e^{\xi_1} (\frac{1}{2} + \delta) \right) \leq \delta^2 \left( 2 + e^{\delta} (\frac{1}{2} +\delta) \right) \leq  \delta$, we deduce that
$$
\alpha(n,\delta, \varepsilon) \leq 4 \delta.
$$
Then, using again that $\delta < 10^{-1}$ (so that $\frac{11}{2} e^\delta \leq 10$),
we have 
\begin{eqnarray*}
\beta(n,\delta, \varepsilon) & = &  \left\{  \Bigl( \sum_{k=1}^{n-1} {\scriptstyle\binom{n-1}{k}} \varepsilon^k \Bigr) + (1 + \varepsilon)^{n-1} \bigl( \varepsilon +  \alpha(n, \delta,\varepsilon) (1+\varepsilon) \bigr) \right\} \\
& \leq &  (e^\delta -1) + e^\delta \bigl( \varepsilon + \alpha(n,\delta, \varepsilon)\bigr) \\ 
& \leq  &  \delta + \frac{e^\delta}{2} \delta^2 + ( 1 + \delta e^{\delta})(  5\delta)\\
&  \leq &  6\delta + \frac{11}{2}\delta^2 e^\delta\\
& \leq & 7 \delta.
\end{eqnarray*}
The result  follows now  directly   from Lemma~\ref{LemmeEstimationProduitScalaire}.
\end{proof}
We can finally state the main result of this section: 
\begin{lem}\label{LemmeEvaluationR}
We suppose that Assumption~\ref{AssumptionDeltavarepsilon} and Assumption~\ref{assumtionDeltaEpsi} hold true. Then,
we have 
$$
\left| \mathscr{R}^{\text{\tt N}}-  \mathscr{R} \right|  \leq   14 \delta n + 5 \delta  | \mathscr{R}^{\text{\tt N}}|.
$$ 
\end{lem}
\begin{proof}
First, we use twice the triangular inequality along with  Lemma~\ref{LemmeEstimationProduitScalaire2} 
to get  
\begin{eqnarray}
\left| \mathscr{R}^{\texttt{N}}-  \mathscr{R} \right| &  \leq &   \left| \mathscr{R}^{\texttt{N}}-  \poscal{Q_{F} x}{x} \right|  + | \poscal{Q_{F} x}{x} | \frac{\val{1- \| x \|^2_{2}}}{\| x \|^2_{2}}  \nonumber \\
& \leq &  \frac{\norm{x}_{2}^{2}+\val{1-\norm{x}_{2}^{2}}}{\| x \|_2^2} \left| \mathscr{R}^{\texttt{N}}-  \poscal{Q_{F} x}{x} \right| +  | \mathscr{R}^{\texttt{N}}  |  \frac{\val{1 - \| x \|_2^2}}{\|x \|_2^2}   \nonumber \\
& \leq & 
\bigl(\norm{x}_{2}^{2}+\val{1-\norm{x}_{2}^{2}}\bigr)
 7 \delta n  +  | \mathscr{R}^{\texttt{N}}  |  \frac{\val{1 - \| x \|_2^2}}{\|x \|_2^2}.
\end{eqnarray}
It is now enough to prove that, under Assumption~\ref{AssumptionDeltavarepsilon},
(in particular \eqref{433})
 and Assumption~\ref{assumtionDeltaEpsi}, 
 we have 
\begin{equation}\label{estimationU2}
\quad  \left |  \frac{1 - \| x \|_2^2}{\|x \|_2^2}  \right|  \leq 5\delta
\quad\text{and}\quad\norm{x}^{2}_{2}+\val{1-\norm{x}^{2}_{2}}\le 2.
\end{equation} 
To get \eqref{estimationU2}
we remark first  that $(\norm{x}_{2}^{2})^{\text{\tt N}}$
corresponds in fact to the number 
$$(\norm{x}_{2}^{2})^{\text{\tt N}}=
S^{\texttt{N}}( P^{\texttt{N}}(x_1, \overline{x_1}), P^{\texttt{N}}(x_2,\overline{x_2}),  \cdots, P^{\texttt{N}}(x_n, \overline{x_n}) ) ,
$$ 
which means that \eqref{433}
in Assumption \ref{AssumptionDeltavarepsilon} can be written as 
$$
\val{S^{\texttt{N}}( P^{\texttt{N}}(x_1, \overline{x_1}), P^{\texttt{N}}(x_2,\overline{x_2}),  \cdots, P^{\texttt{N}}(x_n, \overline{x_n}) ) -1}\le \varepsilon.
$$
But Assumption~\ref{AssumptionDeltavarepsilon} and Lemma~\ref{LemmeProduit} ensure that, for all $j \in \llbracket 1, n \rrbracket$,
we have  
$$
| P^{\texttt{N}}(x_j, \overline{x_j}) - \val{x_j}^2 | \leq \varepsilon |x_j|^2.
$$ 
As a result,  applying Lemma~\ref{Somme_n_nombres} gives 
\begin{multline*}
\left| S^{\texttt{N}}( P^{\texttt{N}}(x_1, \overline{x_1}), P^{\texttt{N}}(x_2,\overline{x_2}),  \cdots, P^{\texttt{N}}(x_n, \overline{x_n}) ) - \|x \|_2^2 \right| 
\\
 \leq \left( \Bigl( \sum_{k=1}^{n-1} \binome{n-1}{k}  \varepsilon^k \Bigr) +  (1 + \varepsilon)^{n-1}\varepsilon \right) \| x \|_2^2 \leq \bigl( (1+\varepsilon)^n -1 \bigr)  \| x \|_2^2.
 \end{multline*}
 Then, the triangular inequality combined with Assumption~\ref{assumtionDeltaEpsi} (see~\eqref{439}) leads to
\begin{multline*}
\left| 1 - \|x \|_2^2 \right|   \leq \left | 1 - (\| x \|_2^2 )^{\texttt{N}}     \right|  + \left | (\| x \|_2^2 )^{\texttt{N}}    - \|x \|_2^2  \right|
\leq  \varepsilon + ( e^\delta  -1 )  \| x \|_2^2   \leq  \varepsilon + 2  \delta   \| x \|_2^2.
 \end{multline*}
 We deduce from the previous  inequality that
 $$
\frac{(1-\varepsilon)}{(1 + 2\delta)} \leq \| x \|_2^2 \leq \frac{(1+\varepsilon)}{(1 - 2 \delta)},\qquad\text{and thus\quad} \val{\norm{x}_{2}^{2}-1}\le \frac{\varepsilon+2\delta}{1-2\delta},
 $$ 
so that 
$$
\left |  \frac{1 - \| x \|_2^2}{\|x \|_2^2}  \right| \leq \frac{(1+ 2\delta)}{(1 - 2 \delta)}  \frac{(2\delta + \varepsilon)}{(1-\varepsilon)}.
$$ 
Since $\delta \leq 0.1$ and $\varepsilon \leq \delta \leq 0.1$, we obtain
$$
\left |  \frac{1 - \| x \|_2^2}{\|x \|_2^2}  \right| \leq  \frac{3}{2}  \frac{10}{9} 3 \delta = 5 \delta,
$$  
and
$$
\norm{x}_{2}^{2}+\val{1-\norm{x}_{2}^{2}}\le
\frac{1+\varepsilon}{1-2\delta}+\frac{\varepsilon+2\delta}{1-2\delta}
\le \frac{1+4\delta}{1-2\delta}=1+\frac{6\delta}{1-2\delta}\le 1+\frac68<2,$$ 
which proves~\eqref{estimationU2}.
 \end{proof} 
\subsection{Numerical results \label{SubsectionResNum}}
To finish this part, it remains to prove Estimates \eqref{EvaluationErreurNumeriqueQuotientRayleigh} and \eqref{EvaluationNumeriqueQuotientRayleigh}. 
\par
First, let  us briefly describe our numerical experiments. In what follows, as indicated in the beginning of this section, we use $F = F_k$  (defined in \eqref{DefinitionFk})   with $k = 70$. As a consequence the matrix $\mathbf{Q}_F$ is a square Hermitian matrix of size $n = 2k + 1= 141$.  For the numerical computation of the Rayleigh quotient $\mathscr{R}$, we use the software Matlab.   The vector $x$ used in the evaluation of the Rayleigh quotient is obtained  using the function \texttt{eig}, taking the eigenvector associated with the largest eigenvalue of $\mathbf{Q}_F$: 
$$
\mathscr{R}^{\texttt{N}} = 1.000070857452742 > 1.00007. 
$$ 
We emphasize that the use of the function \texttt{norm} of Matlab or the direct implementation of the power method~\cite[Chapter 6]{Quarteroni} provide the same results (up to 15 digits  of accuracy), and consequently the same lower bound.
\par
To conclude our proof, it remains to prove the estimate on the error on the Rayleigh quotient $\eqref{EvaluationErreurNumeriqueQuotientRayleigh}$. This is based on the application of Lemma~\ref{LemmeEvaluationR}. To do so, we have to evaluate  $\delta$ (the maximum absolute error made on the coefficients of $\mathbf{Q}_F$) and to verify that Assumption~\ref{AssumptionDeltavarepsilon} and Assumption~\ref{assumtionDeltaEpsi}  are fulfilled. For the first part,  we computed the matrix $\mathbf{Q}_F$ using double and quadruple precision (using a Fortran 90 code). We denote the corresponding matrix $\mathbf{Q}_F^d$ and $\mathbf{Q}_F^q$. We find that
\begin{equation}\label{DeltaQ}
\Delta \mathbf{Q}_{F} = \max_{(i,j) \in \llbracket 1, 141 \rrbracket^2} | (\mathbf{Q}_F^d)_{ij} -(\mathbf{Q}_F^q)_{ij}|  \leq 3 \times  10^{-14}.
\end{equation}
As a result, it is coherent to take  $\delta =10^{-13}$. We can verify that $$
n \varepsilon\le 141\times 6.5\times 10^{-16} \leq 9.2 \times 10^{-14} \leq 10^{-13} =\delta,$$ so that Assumption~\ref{AssumptionDeltavarepsilon} and Assumption~\ref{assumtionDeltaEpsi}  will be fulfilled
if we prove \eqref{ConditionQij}:
to do so, we check Formula
\eqref{2323} and we see that 
$$
\val{a_{j,k}}^{2}\le \frac14+\val{\Phi(k-j)}^{2}.
$$
Since $\Phi$ is odd, we need only to check $\frac14+\val{\Phi(l)}^{2}$ for $l\ge 1$. For $x> 1$, we have 
$$
\Phi'(x)=\frac1{2\pi}\Lg\bigl(\frac{x^{2}-1}{x^{2}}\bigr)<0,
$$
so that 
$ \Phi(x)\le \Phi(1)=\frac{\Lg 2}{\pi},
$
and thus
since $\Phi(x)>0$ for $x>0$ from \eqref{314}, we get for $j\not=k$,
$$
\val{a_{j,k}}^{2}\le \frac14+\frac{(\Lg 2)^{2}}{\pi^{2}}<0.298681.
$$
Moreover, we have for $j\ge 0$, $\val{a_{j,j}}=1/2$ and for $j<0$, $a_{j,j}=0$.
 Applying  Lemma~\ref{LemmeEvaluationR} then gives (noticing that $\delta n \leq 2\times 10^{-11}$ and $ |\mathscr{R}^{\texttt{N}} | < 2$),
$$
\left| \mathscr{R}^{\texttt{N}}-  \mathscr{R} \right|  \leq   2.8\times   10^{-10} + 10^{-12} \leq 10^{-9}, 
$$ 
and~\eqref{EvaluationErreurNumeriqueQuotientRayleigh} is proved.
\begin{rem}
The reader may object that Formula~\eqref{312} used for the evaluation $\Phi(x)$ may lead to cancellation rounding error for  $|x|$ large. Indeed, for $x>0$ large, we have 
\begin{align*}
 (x+1) \ln (x+1)&=(x+1)\ln x+(x+1)\ln\frac{x+1}{x}\\&=(x+1) \ln x+(x+1)\bigl(\frac 1x-\frac{1}{2 x^{2}}+O(x^{-3})\bigr)
 \\&=(x+1) \ln x+1+\frac{1}{2x}+O(x^{-2}),
\end{align*}
as well as
\begin{align*}
 (x-1) \ln (x-1)&=(x-1)\ln x+(x-1)\ln\frac{x-1}{x}\\&=(x-1) \ln x+(x-1)\bigl(-\frac 1x-\frac{1}{2 x^{2}}+O(x^{-3})\bigr)
 \\&=(x-1) \ln x-1+\frac{1}{2x}+O(x^{-2}),
\end{align*}
so that for $x>0$, we have 
\begin{align*}
 2\pi \Phi(x)=&(x+1) \ln (x+1)+ (x-1) \ln (x-1)-2x\ln x
=\frac{1}{x}+O(x^{-2}),
 \end{align*}
leading to the compensation of the large and constant  terms 
$$
(x+1)\ln x+1+(x-1) \ln x-1-2x \ln x=0,
$$
triggering  possibly cancellation errors. 
More specifically, in the present case, for $n=141$, since $ n \ln n \approx 700$ and 
$1/n \approx  7 \times10^{-3}$, we might expect the relative error on $\Phi(x)$ to be bounded by
 $$
\frac{ | (\Phi(x))^N - \Phi(x) | }{|\Phi(x)|} \leq  10^{5} \varepsilon   \leq 7 \times 10^{-11}.
 $$ 
 It then leads to the following bound for  the absolute error on $\Phi$
 $$
| (\Phi(x))^N - \Phi(x) |  \leq ( 7 \times 10^{-11} ) |\Phi(x)| \leq 8 \times 10^{-14},
 $$  
 which is of the same order of magnitude than~\eqref{DeltaQ}.
To overcome this difficulty, we could  write  the formula  as 
  $$
 \Phi(x) = \frac{1}{2\pi} \left(x  \ln\bigl(1-\frac{1}{x^2}\bigr) + \ln\bigl(1 + \frac{2}{x-1}\bigr) \right)  \quad \mbox{for } x >1.
 $$ 
 The latter  formula is apparently more stable numerically since
$$
 x \ln(1 - \frac 1{x^2}) = -\frac{1}{x} + O(x^{-2}), \quad \ln(1 + \frac{2}{x-1}) = \frac{2}{x} + O(x^{-2}),
 $$ 
the leading terms  not compensating each other.  In that case, we observe that 
 $$
 \Delta Q_{F} = \max_{(i,j) \in \llbracket 1, 141 \rrbracket^2} | (Q_F^d)_{ij} -(Q_F^q)_{ij}|  \leq 2\times 10^{-15},
 $$ 
and the choice of $\delta = 10^{-13}$ is obviously valid again. Note that we have obtained the same value for $\mathscr{R}^{\texttt{N}}$ up to 14 digits of accuracy.
\end{rem}
\begin{rem}
We also conducted the full computation of $\mathscr{R}^{\texttt{N}}$ using quadruple precision for $Q_F$ (exporting $x$ obtained with the function \texttt{eig} of Matlab, or programming directly the power method). Here again,  we obtain the same value for $\mathscr{R}^{\texttt{N}}$ up to 14 digits of accuracy.
\end{rem}
\subsection{A summary of numerical results}
We define $\mathcal Q_{k}=\mathbf Q_{F_{k}}$
as defined by
\eqref{DefinitionFk}, Notation \ref{not34} and Proposition \ref{pro5555}.
The matrix $\mathcal Q_{k}$ is a square Hermitian matrix with size $2k+1$, thus with real eigenvalues.
We denote by $$\lambda_{k,1}\ge \lambda_{k,2}\ge \dots\ge  \lambda_{k,2k+1},$$
the eigenvalues of $\mathcal Q_{k}$.
We know from Proposition \ref{bound2},
Lemma \ref{lem.28} and Lemma \ref{lem33}
that
$$
-0.5641896<-\frac1{\sqrt{\pi}}\le \lambda_{k, 2k+1}<0<\lambda_{k,1}\le \frac12+\frac{\sqrt{\pi+1}}{2\sqrt \pi}<
 1.0740884.
$$
\eject
Also we have the following numerical results, for $k$ ranging up to $10^{4}$;
each entry printed in red violates Flandrin's conjecture.
\vs
\vs
\begin{center}
\renewcommand{\arraystretch}{2}
\begin{tabular}{|c|l|l|l|l|l|c|l|}
  \hline\hline\hline
  $\mathbf k$ &\hs $\mathbf{\lambda_{k,1}}$ \hs & \hs $\mathbf{ \lambda_{k,2}}$\hs &\hs  $\mathbf{ \lambda_{k,3}}$\hs &\hs  $\mathbf{ \lambda_{k,4}}$\hs &\hs  $\mathbf{ \lambda_{k,5}}$ \hs&$..$&\footnotesize$\mathbf{ \lambda_{k,2k+1}}$
 \\
\hline\hline\hline
{\tt 3}&0.885305  &0.653839 &0.377158&0.154856&-0.000454&$..$&\footnotesize$ -0.0640857$\\
  \hline
{\tt 5}& 0.936394 & 0.802687&0.615387&0.409291&0.226983&$..$&\footnotesize$-0.0745382$\\
  \hline
{\tt 10}& 0.976219 & 0.926024&0.850022&0.750768&0.634078&$..$&\footnotesize$-0.0866218$\\
  \hline
  {\tt 20}&0.992670&0.976736&0.952903&0.920750&0.880273&$..$&\footnotesize$-0.0963664$\\
  \hline
  {\tt 35}&0.997723&0.991662&0.983057&0.971595&0.957141&$..$&\footnotesize$-0.102900$\\
  \hline
{\tt 70}& {\color{red}\bf  1.00007} & 0.997971&0.995596&0.992540&0.988755&$..$&\footnotesize$-0.109682$\\
  \hline
{\tt 100}& {\color{red}\bf  1.00066}&0.999124& 0.997896& 0.996353& 0.994464&$..$&\footnotesize$-0.112702$ \\
  \hline
{\tt 200}& {\color{red}\bf  1.00149}& 0.999966& 0.999579& 0.999166& 0.998676&$..$&\footnotesize$-0.117815$ \\
  \hline
{\tt 400}& {\color{red}\bf  1.00217}& {\color{red}\bf  1.00028}& 0.999989& 0.999857& 0.999725&$..$&
\footnotesize$
-0.122103$\\
  \hline
  {\tt 800}& {\color{red}\bf  1.00276}& {\color{red}\bf  1.00053}& {\color{red}\bf  1.00014}& {\color{red}\bf  1.00002}& 0.999972&$..$&\footnotesize$-0.125728$\\
  \hline
 {\tt 1000}& {\color{red}\bf  1.00293}& {\color{red}\bf  1.00061}& {\color{red}\bf  1.00018}& {\color{red}\bf  1.00005}& {\color{red}\bf  1.00000}&$..$&\footnotesize$-0.126776$\\
  \hline
{\tt 2000}&{\color{red}\bf  1.00343}& {\color{red}\bf  1.00086}& {\color{red}\bf  1.00030}& {\color{red}\bf  1.00013}&{\color{red}\bf  1.00006} &$..$&\footnotesize$-0.129713$\\
  \hline
  {\tt 3000}&{\color{red}\bf  1.00369}&{\color{red}\bf   1.00101}&{\color{red}\bf   1.00037}&{\color{red}\bf   1.00017}& {\color{red}\bf  1.00009}&$..$&\footnotesize$-0.131236$\\
    \hline
  {\tt 4000}&{\color{red}\bf  1.00386}&{\color{red}\bf 1.00112  }&{\color{red}\bf  1.00043 }&{\color{red}\bf 1.00020  }& {\color{red}\bf 1.00011 }&$..$&\footnotesize$-0.132240$\\
    \hline
  {\tt 6000}&{\color{red}\bf 1.00408}&{\color{red}\bf  1.00127}&{\color{red}\bf 1.00050  }&{\color{red}\bf 1.00025}& {\color{red}\bf 1.00014}&$..$&\footnotesize$-0.133556$\\
   \hline
  {\tt 8000}&{\color{red}\bf 1.00423}&{\color{red}\bf  1.00138}&{\color{red}\bf 1.00056  }&{\color{red}\bf 1.00028}& {\color{red}\bf 1.00016}&$..$&\footnotesize$-0.134426$\\
\hline
{\tt 10 000}&{\color{red}\bf  1.00434}&{\color{red}\bf   1.00147}&{\color{red}\bf   1.00060}&{\color{red}\bf   1.00031}& {\color{red}\bf  1.00018}&$..$&\footnotesize$-0.135068$\\
\hline
\end{tabular}
\end{center}
\vs\vs
We have $\lambda_{70,1}^{\text{\tt N}}=1.000070857452742$ and $\lambda_{100,1}^{\text{\tt N}}=1.00065932861331$.
The paper \cite{MR2131219} by J.G.~Wood \& A.J.~Bracken
studied integrals of the Wigner distribution
on the regions $\{(x,\xi)\in \R^{2}, x>0, x\xi>\alpha\}$, where $\alpha$ is a non-negative parameter
and provided some numerical bounds for the case $\alpha=0$
(which is the quarter-plane),
$$
-0.155 939 843<\lambda_{k,l}  <1.007 679 970.
$$
\begin{figure}[ht]
\hskip-25pt\scalebox{0.8}{\includegraphics[angle=0,height=320pt,width=1.3
\textwidth]{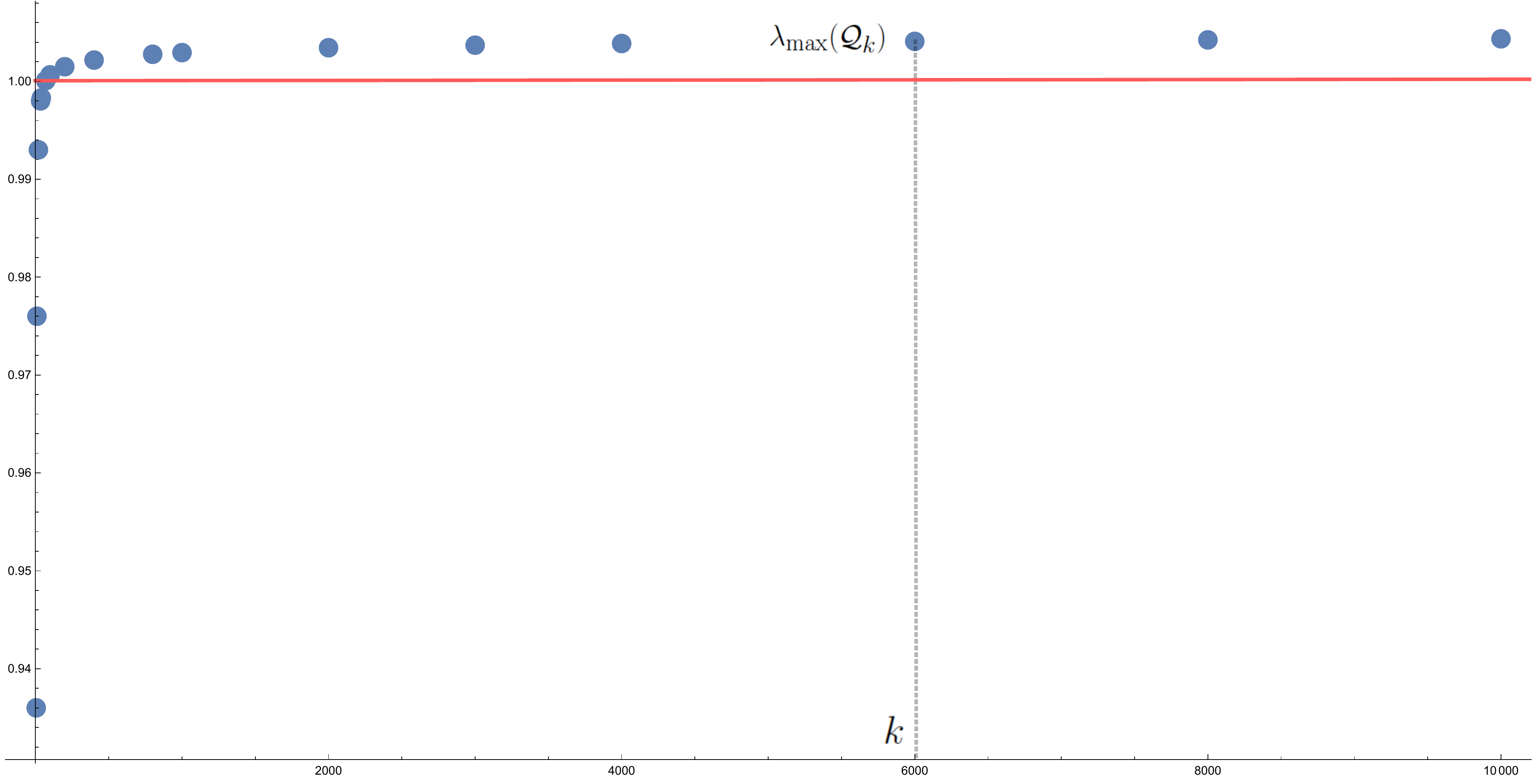}}\par
\caption{\footnotesize The largest eigenvalue of $\mathcal Q_{k}$ is above the threshold {\color{red}$1$} for $k\ge 68.$ }
\end{figure}
\section{Proof of the main result, Further Comments}\label{section5}
\subsection{Proof of Theorem \ref{thmmain}}
Using Theorem \ref{thmnum}, Theorem \ref{thmkey} and \eqref{flandrin++}
we obtain that there exists $u_{0}\in \mathscr S(\R)$ such that 
$$
\iint_{\R_{+}\times \R_{+}} \mathcal W(u_{0},u_{0})(x,\xi) dx d\xi>\norm{u_{0}}_{L^{2}(\R)}^{2}.
$$
Thanks to the results of Section \ref{sec11}, the function $\mathcal W(u_{0},u_{0})$ belongs to $\mathscr S(\R^{2})$,
and the Lebesgue Dominated Convergence Theorem implies that 
$$
\lim_{a\rightarrow+\io}\iint_{[0,a]^{2}} \mathcal W(u_{0},u_{0})(x,\xi) dx d\xi
=\iint_{\R_{+}\times \R_{+}} \mathcal W(u_{0},u_{0})(x,\xi) dx d\xi>\norm{u_{0}}_{L^{2}(\R)}^{2},
$$ 
so that for $a$ large enough,
\begin{equation}\label{singa}
\iint_{[0,a]^{2}} \mathcal W(u_{0},u_{0})(x,\xi) dx d\xi>\norm{u_{0}}_{L^{2}(\R)}^{2},
\end{equation}
concluding the proof of Theorem  \ref{thmmain}.
\subsection{Integrals on convex subsets}
Several interesting questions can be formulated about the integrals of the Wigner distribution on convex sets.
We have seen that for the quarter-plane $C_{0}$ defined in \eqref{quarter}, although the spectrum of 
$\indic{C_{0}}^w$
intersects $(1,+\io)$, we have proven the estimate \eqref{best?} providing an upper bound for that spectrum.
The estimates of Theorem \ref{thmnum} leave a wide gap between the lower bound and the upper bound given by \eqref{best?}.
It would be interesting to know the least upper bound of the spectrum of $\bigl(H(x) H(\xi)\bigr)^{w}$.
\par
We may also consider a general convex polygon $P_{N}$ with $N$ vertices in the plane
and it is possible to prove that there exists $\sigma_{N}>1$ such that 
$$
\indic{P_{N}}^{w}\le \sigma_{N}.
$$ 
For instance, it is possible to prove that for all triangles ($N=3$), we have
$
\indic{P_{3}}^{w}\le 1.7.
$
\par
However, we do not know if the following weakening of Flandrin's conjecture holds true:
is it possible to find $\sigma>1$ such that for all convex subsets $C$ of the plane
$$
\indic{C}^{w}\le \sigma \ ?
$$
Another remark is concerned with the lack of smoothness of the boundary of the cube or of the quarter-plane;
of course starting from \eqref{singa},
we can find an analytic family of open subsets of $[0,a]^{2}$ such as for $p\in 2\N^{*}$,
$$
\Omega_{p}=\{(x,\xi)\in \R^{2}, \val{x-\frac a2}^{p}+\val{\xi-\frac a2}^{p}<\bigl(\frac a 2\bigr)^{p}\},
$$
and since $\mathcal W(u_{0},u_{0})\in \mathscr S(\R^{2})$, we get  for $a$ satisfying \eqref{singa}
$$
\lim_{p\rightarrow+\io}\iint_{\Omega_{p}}  \mathcal W(u_{0},u_{0})(x,\xi) dx d\xi
=\iint_{[0,a]^{2}} \mathcal W(u_{0},u_{0})(x,\xi) dx d\xi>\norm{u_{0}}_{L^{2}(\R)}^{2},
$$
proving that the spectrum of $\indic{\Omega_{p}}^{w}$ intersects $(1,+\io)$ for $p$ large enough,
showing that a counterexample to Flandrin's conjecture can be an analytic open bounded set.
The next result shows that we can produce many examples of convex sets failing to satisfy the estimate
required by the Flandrin conjecture.
We recall a few notions of convex analysis (see e.g. Chapter 1 in \cite{MR0274683}).
Let $K$ be a  closed convex subset of $\R^{n}$; for 
$x_{0}\in\p K$
we define
\begin{equation}\label{520++}
 \mathcal C({x_{0}})=\cup_{\lambda>0} \lambda(K-x_{0}).
\end{equation}
Then $ \mathcal C(x_{0})$ is a convex cone, i.e. is closed under linear combination with  positive coefficients:
indeed for $t_{1},\dots, t_{N}\in  \mathcal C(x_{0})$, $\alpha_{1},\dots, \alpha_{N}$ positive
we have with $\lambda_{1},\dots, \lambda_{N}$ positive and $x_{1},\dots, x_{N}\in K$,
$\Lambda=\sum_{1\le j\le N}\alpha_{j}\lambda_{j}$,
$$
\sum_{1\le j\le N}\alpha_{j}t_{j}=\sum_{1\le j\le N}\alpha_{j}\lambda_{j}(x_{j}-x_{0})=\Lambda\bigl(
\underbrace{\sum_{1\le j\le N}\Lambda^{-1}\alpha_{j}\lambda_{j}x_{j}}_{\in K \text{ by convexity}}-x_{0}\bigr)\in \mathcal C(x_{0}).
$$
We note also that 
\begin{equation}\label{521++}
\text{for $0<\lambda_{1}\le \lambda_{2}$,
\quad}
\lambda_{1}(K-x_{0})\subset\lambda_{2}(K-x_{0}),
\end{equation}
since for $x_{1}\in K$, we have
$
\lambda_{1}(x_{1}-x_{0})=\lambda_{2}(x_{2}-x_{0})
$
with
$$
x_{2}=\frac{\lambda_{1}}{\lambda_{2}} x_{1}+\bigl(1-\frac{\lambda_{1}}{\lambda_{2}}\bigr) x_{0} \in K\text{ by convexity.}
$$
\begin{defi}
 Let $K$ be a compact convex subset of $\R^{2}$ with non-empty interior  and let $x_{0}\in \p K$.
 We shall say that $x_{0}$ is a corner of $K$ if 
  $$
\mathcal C(x_{0})=\{r e^{i\theta}\}_{r>0,\Phi_{0}\le \theta\le \Phi_{0}+\Theta_{0}}, \quad \Theta_{0}\in (0,\pi), \quad \val{\Phi_{0}}<\pi.
$$
\end{defi}
\begin{pro}
 Let $K$ be a compact convex subset of $\R^{2}$ such that there exists $X_{0}\in \p K$
 which is a corner of $K$.
  Then there exists $\mu>0$ such that  the compact convex set 
 $$
K_{\mu}=X_{0}+\mu(K-X_{0}),
 $$
 is such that 
 the spectrum of $\indic{K_{\mu}}^{w}$ intersects $(1,+\io)$.
\end{pro}
\begin{nb}
By symplectic invariance, $K_{\mu}$ can be replaced by $X_{1}+\mu K$ for any $X_{1}\in \R^{2}$.
\end{nb}
\begin{proof}
  Using a translation, a rotation in the plane and their quantizations, we may assume that 
$X_{0}=0$,
$$
\mathcal C(X_{0})=\mathcal L_{\Theta_{0}}
=\{r e^{i\theta}\}_{\substack{r>0\\0\le \theta\le \Theta_{0}}}, \ \Theta_{0}\in (0,\pi),
\quad
\cup_{\mu>0} K_{\mu}=X_{0}+\mathcal C(X_{0})=\mathcal L_{\Theta_{0}}.
$$
As a result from \eqref{520++}, \eqref{521++}, we have for $\mu>0$ and $u\in \mathscr S(\R)$,
\begin{equation}\label{523--}
\lim_{\mu\rightarrow+\io}\iint \indic{K_{\mu}}(x,\xi)\mathcal W(u,u)(x,\xi) dx d\xi=
\iint \indic{\mathcal L_{\Theta_{0}}}(x,\xi)\mathcal W(u,u)(x,\xi) dx d\xi.
\end{equation}
Using the symplectic matrix
$$S_{0}=
\begin{pmatrix}
1&-\frac{\cos \Theta_{0}}{\sin\Theta_{0}}
\\
0&1
\end{pmatrix},
$$
we see that 
$
S_{0}(\mathcal L_{\Theta_{0}})=C_{0}
$
where $C_{0}$ is the quarter-plane.
As a result, the operator $\indic{C_{0}}^{w}$ is unitarily equivalent to 
$\indic{\mathcal L_{\Theta_{0}}}^{w}$ which   thus has a spectrum intersecting $(1,+\io)$ and there exists $u_{1}\in \mathscr S(\R)$ such that
$$
\iint \indic{\mathcal L_{\Theta_{0}}}(x,\xi)\mathcal W(u_{1},u_{1})(x,\xi) dx d\xi>\norm{u_{1}}_{L^{2}(\R)}^{2},
$$
implying from \eqref{523--}
that for $\mu$ large enough,
$$
\iint \indic{K_{\mu}}(x,\xi)\mathcal W(u_{1},u_{1})(x,\xi) dx d\xi>\norm{u_{1}}_{L^{2}(\R)}^{2},
$$
proving
the proposition.
\end{proof}
\subsection{Further comments}
A more difficult problem related to the initial question by P.~Flandrin
would be to find a geometric condition on a compact subset $K$ of the plane to ensure that
\begin{equation}\label{5est}
\indic{K}^{w}\le \Id.
\end{equation}
We have seen that convexity of $K$ is not enough for that property to hold true, but convexity is not necessary either:
simple examples are for $K$ with a Lebesgue measure smaller than 1/2, thanks to the first estimate of \eqref{norm01}, but also some non-convex sets with large Lebesgue measure may satisfy \eqref{5est}:
in fact using Flandrin's estimate \eqref{134}, we find that for any $a\ge 0$, we have 
$$
\indic{D_{a}}^{w}\le 1-e^{-a},
$$
so that we may consider $D_{a}\cup M_{a}$ where $M_{a}$ is any subset of the plane with Lebesgue measure smaller than $e^{-a}/2$ and get
$$
\indic{D_{a}\cup M_{a}}^{w}\le 1,
$$
without convexity for $D_{a}\cup M_{a}$.
\vs

\end{document}